\documentclass[10pt,twoside]{amsart}
\usepackage{amsmath,amsfonts,amssymb,amsxtra,setspace,xspace,graphicx,lmodern,psfrag,epsfig,color,latexsym}
\usepackage[T1]{fontenc}
\usepackage[colorlinks=true]{hyperref}\hypersetup{urlcolor=blue, citecolor=red}

\newtheorem{thm}{Theorem}

\newtheorem{cor}[thm]{Corollary}
\newtheorem{prop}[thm]{Proposition}
\newtheorem{rem}[thm]{Remark}

\def\eps{\varepsilon}

\newcommand{\R}{{\mathbb R}}

\def\eps{\varepsilon}

\begin{document}

\title[Cloaking via mapping for the heat equation]{Cloaking via mapping for the heat equation}

\bibliographystyle{alpha}

\author{R. V. Craster}
\address{R.V.C.: Department of Mathematics, Imperial College London, London, SW7 2AZ, United Kingdom.}
\email{r.craster@imperial.ac.uk}

\author{S. R. L. Guenneau}
\address{S.R.L.G.: Aix-Marseille Universite, CNRS, Centrale Marseille, Institut Fresnel, Avenue Escadrille Normandie-Ni\'emen, 13013, Marseille, France.}
\email{sebastien.guenneau@fresnel.fr}

\author{H. R. Hutridurga}
\address{H.R.H.: Department of Mathematics, Imperial College London, London, SW7 2AZ, United Kingdom.}
\email{h.hutridurga-ramaiah@imperial.ac.uk}

\author{G. A. Pavliotis}
\address{G.A.P.: Department of Mathematics, Imperial College London, London, SW7 2AZ, United Kingdom.}
\email{g.pavliotis@imperial.ac.uk}

\date{\today}

\maketitle

\setcounter{tocdepth}{1}
\tableofcontents

\thispagestyle{empty}

\begin{abstract}
This paper explores the concept of near-cloaking in the context of time-dependent heat propagation. We show that after the lapse of a certain threshold time instance, the boundary measurements for the homogeneous heat equation are close to the cloaked heat problem in a certain Sobolev space norm irrespective of the density-conductivity pair in the cloaked region. A regularised transformation media theory is employed to arrive at our results. Our proof relies on the study of the long time behaviour of solutions to the parabolic problems with high contrast in density and conductivity coefficients. It further relies on the study of boundary measurement estimates in the presence of small defects in the context of steady conduction problem. We then present some numerical examples to illustrate our theoretical results.
\end{abstract}

\section{Introduction}

\subsection{Physical motivation}
This work addresses the concept of near-cloaking in the context of the
time-dependent heat propagation. Our study is motivated by experiments
in electrostatics \cite{Liu_2012} and thermodynamics
\cite{Schittny_2013}, which have demonstrated the markedly different
behaviour of structured cloaks in static and dynamic regimes. It has
been observed, in the physics literature, that the thermal field in
\cite{Schittny_2013} reaches an equilibrium state after a certain time
interval, when it looks nearly identical to electrostatic field
measured in \cite{Liu_2012}. There have already been some rigorous results for the electrostatic case
\cite{Greenleaf_2003}. However, during the transient regime the
thermal field looks much different from the static result, and a
natural question that arises is whether one can give a mathematically
rigorous definition of cloaking for diffusion processes in the time
domain. This has important practical applications as cloaking for
diffusion processes have been applied to mass transport problems in
life sciences \cite{Guenneau_2013} and chemistry \cite{Zeng_2013} as
well as to multiple light scattering whereby light is governed by
ballistic laws \cite{Schittny_2014}. Interestingly, the control of wave
trajectories first proposed in the context of electromagnetism
\cite{pendry_2006} can also be extended to matter waves that are
solutions of a Schr\" odinger equation \cite{Zhang_2008} that is akin to the
heat equation and so the near-cloaking we investigate has broad
implications and application. 

\subsection{Analytical motivation}
Change-of-variables based cloaking schemes have been inspired by the work of Greenleaf and co-authors \cite{Greenleaf_2003} in the context of electric impedance tomography and by the work of Pendry and co-authors \cite{pendry_2006} in the context of time-harmonic Maxwell's equations. The transformation employed in both those works is singular and any mathematical analysis involving them becomes quite involved. The transformation in \cite{Greenleaf_2003, pendry_2006} essentially blows up a point to a region in space which needs to be cloaked. These works yield perfect-cloaks, i.e., they render the target region completely invisible to boundary measurements. Regularised versions of this singular approach have been proposed in the literature. In \cite{Kohn_2008}, Kohn and co-authors proposed a regularised approximation of this map by blowing up a small ball to the cloaked region and studied the asymptotic behaviour as the radius of the small ball vanishes, thus recovering the singular transform of \cite{Greenleaf_2003, pendry_2006}. An alternate approach involving the truncation of singularities was employed by Greenleaf and co-authors in \cite{Greenleaf_2008} to provide an approximation scheme for the singular transform in \cite{Greenleaf_2003, pendry_2006}. It is to be noted that the constructions in \cite{Kohn_2008} and \cite{Greenleaf_2008} are shown to be equivalent in \cite{Kocyigit_2013}. We refer the interested reader to the review papers \cite{Greenleaf_2009A, Greenleaf_2009B} for further details on cloaking via change of variables approach with emphasis on the aforementioned singular transform.

Rather than employing the singular transformation, we follow the lead of Kohn and co-authors \cite{Kohn_2008}. This is in contrast with some works in the literature on the time-dependent thermal cloaking strategies where singular schemes are used e.g., \cite{Guenneau_2012, Petiteau_2014, Petiteau_2015}. Note that the evolution equation which we consider is a good model for \cite{Schittny_2013}, which designs and fabricates a microstructured thermal cloak that molds the heat flow around an object in a metal plate. We refer the interested reader to the review paper \cite{Raza_2016} for further details on transformation thermodynamics. The work of Kohn and co-authors \cite{Kohn_2008} estimates that the near-cloaking they propose for the steady conduction problem is $\eps^d$-close to the perfect cloak, where $d$ is the space dimension. Our construction of the cloaking structure is exactly similar to the construction in \cite{Kohn_2008}. In the present time-dependent setting, we allow the solution to evolve in time until it gets close to the associated thermal equilibrium state which solves a steady conduction problem. This closeness is studied in the Sobolev space $\mathrm H^1(\Omega)$. We then employ the $\eps^d$-closeness result of \cite{Kohn_2008} to deduce our near-cloaking theorem in the time-dependent setting. To the best of our knowledge, this is the first work to consider near-cloaking strategies to address time-dependent heat conduction problem.

In the literature, there have been numerous works on the approximate cloaking strategies for the Helmoltz equation \cite{Kohn_2010, Nguyen_2010, Nguyen_2011, Nguyen_2012} (see also \cite{Nguyen-Vogelius_2012} for the treatment of the full wave equation). The strategy in \cite{Nguyen-Vogelius_2012} to treat the time-dependent wave problem is to take the Fourier transform in the time variable. This yields a family of Helmholtz problems (family indexed by the frequency). The essential idea there is to obtain appropriate degree of invisibility estimates for the Helmholtz equation -- the estimates being frequency-dependent. More importantly, these estimates blow-up in the frequency regime. But, they do so in an integrable fashion. Equipped with these approximate cloaking results for the Helmholtz equations, the authors in \cite{Nguyen-Vogelius_2012} simply invert the Fourier transform to read off the near-cloaking result for the time-dependent wave equation. Inspired by \cite{Nguyen-Vogelius_2012}, one could apply the Laplace transform in the time variable for the heat equation and try to mimic the analysis in \cite{Nguyen-Vogelius_2012} for the family of elliptic problems thus obtained. Note that this approach does not require, unlike ours, the solution to the heat conduction problem to reach equilibrium state to obtain approximate cloaking. We have not explored this approach in detail and we leave it for future investigations. 

Gralak and co-authors \cite{Gralak_2016} recently developed invisible layered structures in transformation optics in the temporal regime. Inspired by that work, we have also developed a transformation media theory for thermal layered cloaks, which are of practical importance in thin-film solar cells for energy harvesting in photovoltaic industry. These layered cloaks might be of importance in thermal imaging. We direct the interested readers to \cite{Ammari_2005} and references therein.

In the applied mathematics community working on meta-materials, enhancement of near-cloaking is another topic which has been addressed in the literature. Loosely speaking, these enhancement techniques involve covering a small ball of radius $\eps$ by multiple coatings and then applying the push-forward maps of \cite{Kohn_2008}. These multiple coatings which result in the vanishing of certain polarization tensors help us improve the $\eps^d$-closeness of \cite{Kohn_2008} to $\eps^{dN}$-closeness where $N$ denotes the number of coatings in the above construction. For further details, we direct the readers to \cite{Ammari_2012a, Ammari_2012b} in the mathematics literature and \cite{Alu_2005, Alu_2007} in the physics literature (the works \cite{Alu_2005, Alu_2007} employ negative index materials). One could employ the constructions of \cite{Ammari_2012a, Ammari_2012b} in the time-independent setting to our temporal setting to obtain enhanced near-cloaking structures. This again is left for future investigations.

In this present work, we are able to treat time-independent sources for the heat equation. As our approach involves the study of thermal equilibration, we could extend our result to time-dependent sources which result in equilibration. This, however, leaves open the question of near-cloaking for the heat equation with genuinely time-dependent sources which do not result in thermal equilibration. For example, sources which are time harmonic cannot be treated by the approach of this paper. The approach involving Laplace transform mentioned in a previous paragraph might be of help here. There are plenty of numerical works published by physicists on these aspects, but with no mathematical foundation thus far. The authors plan to return to these questions in the near future.

\subsection{Paper structure} The paper is organized as follows. In section \ref{sec:math-set}, we briefly recall the change-of-variable-principle for the heat equation. This section also makes precise the notion of near-cloaking and its connection to perfect cloaking followed by the construction of cloaking density and conductivity coefficients. Our main result (Theorem \ref{thm:near-cloak}) is stated in that section. Section \ref{sec:spec} deals with the long time behaviour of solutions to parabolic problems; the effect of high contrast in density and conduction on the long time behaviour of solutions is also treated in that section. The proof of Theorem \ref{thm:near-cloak} is given in section \ref{sec:near-cloak-proof}. In this section, we also develop upon an idea of layered cloak inspired by the construction in \cite{Gralak_2016}. Finally, in section \ref{sec:numerics}, we present some numerical examples to illustrate our theoretical results.

\section{Mathematical setting}\label{sec:math-set}
Let $\Omega\subset\R^d$ $(d=2,3)$ be a smooth bounded domain such that $B_2\subset \Omega$. Throughout, we use the notation $B_r$ to denote an euclidean ball of radius $r$ centred at the origin. 
\subsection{Change-of-variable principle}
The following result recalls the principle behind the change-of-variables based cloaking strategies. This is the typical and essential ingredient of any cloaking strategies in transformation media theory.
\begin{prop}\label{prop:change-variable-princ}
Let the coefficients $A\in\mathrm L^\infty(\Omega;\R^{d\times d})$ and $\rho\in\mathrm L^\infty(\Omega;\R)$. Suppose the source term $f\in\mathrm L^2(\Omega;\R)$. Consider a smooth invertible map $\mathbb{F}:\Omega\mapsto\Omega$ such that $\mathbb{F}(x) = x$ for each $x\in \Omega\setminus B_2$. Furthermore, assume that the associated Jacobians satisfy ${\rm det}(D\mathbb{F})(x), {\rm det}(D\mathbb{F}^{-1})(x) \ge C >0$ for a.e. $x\in\Omega$. Then $u(t,x)$ is a solution to 
\begin{align*}
\rho(x) \frac{\partial u}{\partial t} = \nabla \cdot \Big( A(x) \nabla u \Big) + f(x) \qquad \mbox{ for }(t,x)\in(0,\infty)\times\Omega
\end{align*}
if and only if $v= u\circ \mathbb{F}^{-1}$ is a solution to
\begin{align*}
\mathbb{F}^*\rho(y) \frac{\partial v}{\partial t} = \nabla \cdot \Big( \mathbb{F}^*A(y) \nabla v \Big) + \mathbb{F}^*f(y) \qquad \mbox{ for }(t,y)\in(0,\infty)\times\Omega
\end{align*}
where the coefficients are given as
\begin{equation}\label{eq:push-forward-formulae}
\begin{aligned}
& \mathbb{F}^*\rho(y) = \frac{\rho(x)}{{\rm det }(D\mathbb{F})(x)};
\qquad \qquad \qquad \qquad
& \mathbb{F}^*f(t,y) = \frac{f(t,x)}{{\rm det }(D\mathbb{F})(x)};
\\[0.2 cm]
& \mathbb{F}^*A(y) = \frac{D\mathbb{F}(x) A(x) D\mathbb{F}^\top(x)}{{\rm det }(D\mathbb{F})(x)}
& ~ 
\end{aligned}
\end{equation}
with the understanding that the right hand sides in \eqref{eq:push-forward-formulae} are computed at $x=\mathbb{F}^{-1}(y)$. Moreover we have for all $t>0$,
\begin{align}\label{eq:prop:change-variable-assertion}
u(t,\cdot) = v(t,\cdot) \qquad \mbox{ in }\Omega\setminus B_2.
\end{align}
\end{prop}
The proof of the above proposition has appeared in the literature in most of the papers on ``cloaking via mapping'' techniques. It essentially involves performing a change of variables in the weak formulation associated with the differential equation. We will skip the proof and refer the reader to \cite[subsection 2.2]{Kohn_2008}, \cite[subsection 2.2, page 976]{Kohn_2010}, \cite[section 2, page 8209]{Guenneau_2012} for essential details.\\

Following Kohn and co-authors \cite{Kohn_2008}, we fix a regularising parameter $\eps>0$ and consider a Lipschitz map $\mathcal{F}_\eps:\Omega\mapsto\Omega$ defined below
\begin{equation}\label{eq:std-near-cloak}
\mathcal{F}_\eps(x) :=
\left\{
\begin{array}{cl}
x & \mbox{ for } x\in \Omega\setminus B_2
\\[0.2 cm]
\left( \frac{2-2\eps}{2-\eps} + \frac{\vert x \vert}{2-\eps}\right) \frac{x}{\vert x\vert} & \mbox{ for } x\in B_2 \setminus B_\eps
\\[0.2 cm]
\frac{x}{\eps} & \mbox{ for }x\in B_\eps.
\end{array}\right.
\end{equation}
Note that $\mathcal{F}_\eps$ maps $B_\eps$ to $B_1$ and the annulus
$B_2\setminus B_\eps$ to $B_2\setminus B_1$. The cloaking strategy with the above map corresponds to having $B_1$ as the cloaked region and the annulus $B_2\setminus B_1$ as the cloaking annulus. The Lipschitz map given above is borrowed from \cite[page 5]{Kohn_2008}. Remark that taking $\eps=0$ in \eqref{eq:std-near-cloak} yields the map
\begin{equation}\label{eq:std-singular-cloak}
\mathcal{F}_0(x) :=
\left\{
\begin{array}{cl}
x & \mbox{ for } x\in \Omega\setminus B_2
\\[0.2 cm]
\left( 1 + \frac12 \left\vert x \right\vert \right) \frac{x}{\vert x\vert} & \mbox{ for } x\in B_2 \setminus \{0\}
\end{array}\right.
\end{equation}
which is the singular transform of \cite{Greenleaf_2003, pendry_2006}. The map $\mathcal{F}_0$ is smooth except at the point $0$. It maps $0$ to $B_1$ and $B_2\setminus\{0\}$ to $B_2\setminus B_1$.

\subsection{Essential idea of the paper}
Let us make precise the notion of near-cloaking we will use throughout this paper. Let $f\in\mathrm L^2(\Omega)$ denote a source term such that ${\rm supp}\, f\subset \Omega\setminus B_2$. Let $g\in\mathrm L^2(\partial\Omega)$ denote a Neumann boundary data. Suppose the initial datum $u^{\rm in}\in \mathrm H^1(\Omega)$ be such that ${\rm supp}\, u^{\rm in}\subset \Omega\setminus B_2$. Consider the homogeneous (conductivity being unity) heat equation for the unknown $u_{\rm hom}(t,x)$ with the aforementioned data.
\begin{equation}\label{eq:intro-u_hom}
\begin{aligned}
\partial_t u_{\rm hom} (t,x) & = \Delta u_{\rm hom}(t,x) + f(x) \qquad \mbox{ in }(0,\infty)\times\Omega,
\\[0.2 cm]
\nabla u_{\rm hom} \cdot {\bf n}(x) & = g(x) \qquad \qquad\qquad \qquad\mbox{ on }(0,\infty)\times\partial\Omega,
\\[0.2 cm]
u_{\rm hom}(0,x) & = u^{\rm in}(x) \qquad \qquad\qquad \qquad\qquad \mbox{ in }\Omega.
\end{aligned}
\end{equation}
Here ${\bf n}(x)$ is the unit exterior normal to $\Omega$ at $x\in\partial\Omega$.\\
Our objective is to construct coefficients $\rho_{\rm cl}(x)$ and $A_{\rm cl}(x)$ such that
\[
\rho_{\rm cl}(x) = \eta(x); \qquad
A_{\rm cl}(x) = \beta(x) \qquad \mbox{ in }B_1
\]
for some arbitrary bounded positive density $\eta$ and for some arbitrary bounded positive definite conductivity $\beta$. This construction should further imply that the evolution for the unknown $u_{\rm cl}(t,x)$ given by 
\begin{equation}\label{eq:intro-u_cloak}
\begin{aligned}
\rho_{\rm cl}(x) \partial_t u_{\rm cl} & = \nabla \cdot \Big( A_{\rm cl}(x) \nabla u_{\rm cl} \Big) + f(x) \qquad \mbox{ in }(0,\infty)\times\Omega,
\\[0.2 cm]
\nabla u_{\rm cl} \cdot {\bf n}(x) & = g(x) \qquad \qquad\qquad \qquad \qquad \mbox{ on }(0,\infty)\times\partial\Omega,
\\[0.2 cm]
u_{\rm cl}(0,x) & = u^{\rm in}(x) \qquad \qquad\qquad \qquad\qquad \mbox{ in }\Omega,
\end{aligned}
\end{equation}
is such that there exists a time instant $\mathit{T}<\infty$ so that for all $t\ge \mathit{T}$, we have
\begin{align*}
u_{\rm cl}(t,x) \approx u_{\rm hom}(t,x)
\qquad \mbox{ for }x\in \Omega\setminus B_2.
\end{align*}
The above closeness will be measured in some appropriate function space norm. Most importantly, this approximation should be independent of the density-conductivity pair $\eta,\beta$ in $B_1$. Note, in particular, that the source terms $f(x), g(x), u^{\rm in}(x)$ in \eqref{eq:intro-u_hom} and \eqref{eq:intro-u_cloak} are the same.
\subsection{Cloaking coefficients \& the defect problem}
The following construct using the push-forward maps is now classical
in transformation optics and we use it here for thermodynamics.
\begin{equation}\label{eq:rho-cloak-choice}
\rho_{\rm cl}(x) =
\left\{
\begin{array}{ll}
1 & \quad \mbox{ for }x\in\Omega\setminus B_2,\\[0.2 cm]
\mathcal{F}^*_\eps 1 & \quad \mbox{ for }x\in B_2\setminus B_1,\\[0.2 cm]
\eta(x) & \quad \mbox{ for }x\in B_1
\end{array}\right.
\end{equation}
and
\begin{equation}\label{eq:A-cloak-choice}
A_{\rm cl}(x) =
\left\{
\begin{array}{ll}
{\rm Id} & \quad \mbox{ for }x\in\Omega\setminus B_2,\\[0.2 cm]
\mathcal{F}^*_\eps {\rm Id} & \quad \mbox{ for }x\in B_2\setminus B_1,\\[0.2 cm]
\beta(x) & \quad \mbox{ for }x\in B_1.
\end{array}\right.
\end{equation}
The density coefficient $\eta(x)$ in \eqref{eq:rho-cloak-choice} is any arbitrary real coefficient such that
\[
0 < \eta(x) < \infty \qquad \mbox{ for }x\in B_1.
\]
The conductivity coefficient $\beta(x)$ in \eqref{eq:A-cloak-choice} is any arbitrary bounded positive definite matrix, i.e., there exist positive constants $\kappa_1$ and $\kappa_2$ such that
\[
\kappa_1 \left\vert \xi \right\vert^2 \le \beta(x) \xi \cdot \xi \le \kappa_2 \left\vert \xi \right\vert^2
\qquad \forall (x,\xi)\in B_1\times\R^d.
\]
The following observation is crucial for the analysis to follow. Consider the density-conductivity pair
\begin{equation}\label{eq:rho-small-inclusion}
\rho^\eps(x) =
\left\{
\begin{array}{ll}
1 & \quad \mbox{ for }x\in\Omega\setminus B_\eps,\\[0.2 cm]
\frac{1}{\eps^{d}}\eta\left(\frac{x}{\eps}\right) & \quad \mbox{ for }x\in B_\eps
\end{array}\right.
\end{equation}
and
\begin{equation}\label{eq:A-small-inclusion}
A^\eps(x) =
\left\{
\begin{array}{ll}
{\rm Id} & \quad \mbox{ for }x\in\Omega\setminus B_\eps,\\[0.2 cm]
\frac{1}{\eps^{d-2}}\beta\left(\frac{x}{\eps}\right) & \quad \mbox{ for }x\in B_\eps.
\end{array}\right.
\end{equation}
Next let us compute their push-forwards using the Lipschitz map $\mathcal{F}^\eps$ -- as given by the formulae \eqref{eq:push-forward-formulae} -- yielding
\begin{align*}
\rho_{\rm cl}(y) = \mathcal{F}^*_\eps\rho^\eps(y);
\qquad
A_{\rm cl}(y) = \mathcal{F}^*_\eps A^\eps(y).
\end{align*}
Then, the assertion of Proposition \ref{prop:change-variable-princ} (see in particular the equality \eqref{eq:prop:change-variable-assertion}) implies that the solution $u_{\rm cl}(t,x)$ to \eqref{eq:intro-u_cloak} satisfies for all $t>0$,
\begin{align*}
u_{\rm cl}(t,x) = u^\eps(t,x) 
\qquad \mbox{ for all }x\in \Omega\setminus B_2
\end{align*}
with $u^\eps(t,x)$ being the solution to 
\begin{align}\label{eq:intro:small-inclusion}
\begin{aligned}
\rho^\eps(x) \partial_t u^\eps & = \nabla \cdot \Big( A^\eps(x) \nabla u^\eps \Big) + f(x) \qquad \mbox{ in }(0,\infty)\times\Omega,
\\[0.2 cm]
\nabla u^\eps \cdot {\bf n}(x) & = g(x) \qquad \qquad\qquad \qquad \qquad \mbox{ on }(0,\infty)\times\partial\Omega,
\\[0.2 cm]
u^\eps(0,x) & = u^{\rm in}(x) \qquad \qquad\qquad \qquad\qquad \mbox{ in }\Omega,
\end{aligned}
\end{align}
with the coefficients in the above evolution being given by \eqref{eq:rho-small-inclusion}-\eqref{eq:A-small-inclusion}. The coefficients $\rho^\eps$ and $A^\eps$ are uniform except for their values in $B_\eps$. Hence we treat $B_\eps$ as a defect where the coefficients show high contrast with respect to their values elsewhere in the domain. Due to the nature of these coefficients, we call the evolution problem \eqref{eq:intro:small-inclusion} the \emph{defect problem with high contrast coefficients} or \emph{defect problem} for short. The change-of-variable principle (see Proposition \ref{prop:change-variable-princ}) essentially says that, to study cloaking for the transient heat transfer problem, we need to compare the solution $u^\eps(t,x)$ to the defect problem \eqref{eq:intro:small-inclusion} with the solution $u_{\rm hom}(t,x)$ to the homogeneous problem \eqref{eq:intro-u_hom} for $x\in\Omega\setminus B_2$.
\subsection{Main result}
We are now ready to state the main result of this work.
\begin{thm}\label{thm:near-cloak}
Let the dimension $d\ge2$. Let $u^\eps(t,x)$ be the solution to the defect problem \eqref{eq:intro:small-inclusion} and let $u_{\rm hom}(t,x)$ be the solution to the homogeneous conductivity problem \eqref{eq:intro-u_hom}. Suppose the data in \eqref{eq:intro-u_hom} and \eqref{eq:intro:small-inclusion} are such that
\begin{align*}
f\in\mathrm L^2(\Omega), \quad {\rm supp}\,f\subset \Omega\setminus B_2, \quad g\in\mathrm L^2(\partial\Omega), \quad u^{\rm in}\in\mathrm H^1(\Omega), \quad {\rm supp}\, u^{\rm in}\subset \Omega\setminus B_2.
\end{align*}
Let us further suppose that the source terms satisfy
\begin{align}
\int_\Omega f(x)\, {\rm d}x + \int_{\partial\Omega} g(x)\, {\rm d}\sigma(x) & = 0, \label{eq:thm-near-cloak-compatible}
\\
\int_\Omega u^{\rm in}(x)\, {\rm d}x & = 0. \label{eq:thm-near-cloak-compatible-initial}
\end{align}
Then, there exists a time $\mathit{T} <\infty$ such that for all $t\ge \mathit{T}$ we have
\begin{align}\label{eq:thm:H12-estimate}
\left\Vert u^\eps(t,\cdot) - u_{\rm hom}(t,\cdot) \right\Vert_{\mathrm H^{\frac12}(\partial\Omega)} \le C \left( \left\Vert u^{\rm in} \right\Vert_{\mathrm H^1} + \left\Vert f \right\Vert_{\mathrm L^2(\Omega)} + \left\Vert g \right\Vert_{\mathrm L^2(\partial\Omega)} \right)\eps^d,
\end{align}
where the positive constant $C$ depends on the domain $\Omega$ and the $\mathrm L^\infty$ bounds on the density-conductivity pair $(\eta, \beta)$ in $B_1$.
\end{thm}
Thanks to the change-of-variable principle (Proposition \ref{prop:change-variable-princ}), we deduce the following corollary.
\begin{cor}\label{cor:diff-cloak-hom}
Let $u_{\rm cl}(t,x)$ be the solution to the thermal cloak problem \eqref{eq:intro-u_cloak} with the cloaking coefficients $\rho_{\rm cl}(x), A_{\rm cl}(x)$ given by \eqref{eq:rho-cloak-choice}-\eqref{eq:A-cloak-choice}. Let $u_{\rm hom}(t,x)$ be the solution to the homogeneous conductivity problem \eqref{eq:intro-u_hom}. Suppose the data in \eqref{eq:intro-u_hom} and \eqref{eq:intro-u_cloak} are such that
\begin{align*}
f\in\mathrm L^2(\Omega), \quad {\rm supp}\,f\subset \Omega\setminus B_2, \quad g\in\mathrm L^2(\partial\Omega), \quad u^{\rm in}\in\mathrm H^1(\Omega), \quad {\rm supp}\, u^{\rm in}\subset \Omega\setminus B_2.
\end{align*}
Let us further suppose that the source terms satisfy \eqref{eq:thm-near-cloak-compatible} and \eqref{eq:thm-near-cloak-compatible-initial}. Then, there exists a time $\mathit{T} <\infty$ such that for all $t\ge \mathit{T}$ we have
\begin{align}\label{eq:cor:H12-estimate}
\left\Vert u_{\rm cl}(t,\cdot) - u_{\rm hom}(t,\cdot) \right\Vert_{\mathrm H^{\frac12}(\partial\Omega)} \le C \left( \left\Vert u^{\rm in} \right\Vert_{\mathrm H^1} + \left\Vert f \right\Vert_{\mathrm L^2(\Omega)} + \left\Vert g \right\Vert_{\mathrm L^2(\partial\Omega)} \right) \eps^d,
\end{align}
where the positive constant $C$ depends on the domain $\Omega$ and the $\mathrm L^\infty$ bounds on the density-conductivity pair $(\eta, \beta)$ in $B_1$.
\end{cor}
Whenever the source terms $f(x), g(x)$ and the initial datum $u^{\rm in}(x)$ satisfy the compatibility conditions \eqref{eq:thm-near-cloak-compatible}-\eqref{eq:thm-near-cloak-compatible-initial}, we shall call them \emph{admissible}. Our proof of the Theorem \ref{thm:near-cloak} relies upon two ideas
\begin{enumerate}
\item[$(i)$] long time behaviour of solutions to parabolic problems.
\item[$(ii)$] boundary measurement estimates in the presence of small inhomogeneities.
\end{enumerate}

\begin{rem}
Our strategy of proof goes via the study of the steady state problems associated with the evolutions \eqref{eq:intro-u_hom} and \eqref{eq:intro:small-inclusion}. Those elliptic boundary value problems demand that the source terms be compatible to guarantee existence and uniqueness of solution. These compatibility conditions on the bulk and Neumann boundary sources $f,g$ translates to the zero mean assumption \eqref{eq:thm-near-cloak-compatible}. Note that the compatibility condition \eqref{eq:thm-near-cloak-compatible} is not necessary for the well-posedness of the transient problems.
\end{rem}
A remark about the assumption \eqref{eq:thm-near-cloak-compatible-initial} on the initial datum $u^{\rm in}$ in Theorem \ref{thm:near-cloak} will be made later after the proof of Theorem \ref{thm:near-cloak} as it will be clearer to the reader.

\section{Study of the long time behaviour}\label{sec:spec}
In this section, we deal with the long time asymptotic analysis for the parabolic problems. Consider the initial-boundary value problem for the unknown $v(t,x)$.
\begin{equation}\label{eq:spec-ibvp-v}
\begin{aligned}
\partial_t v & = \Delta v \qquad \mbox{ in }(0,\infty)\times\Omega,
\\[0.2 cm]
\nabla v \cdot {\bf n}(x) & = 0 \quad \qquad \qquad\mbox{ on }(0,\infty)\times\partial\Omega,
\\[0.2 cm]
v(0,x) & = v^{\rm in}(x) \qquad \qquad\qquad \mbox{ in }\Omega.
\end{aligned}
\end{equation}
We give an asymptotic result for the solution to \eqref{eq:spec-ibvp-v} in the $t\to\infty$ limit.
\begin{prop}\label{prop:v-H1-t-infty}
Let $v(t,x)$ be the solution to the initial-boundary value problem \eqref{eq:spec-ibvp-v}. Suppose the initial datum $v^{\rm in}\in\mathrm H^1(\Omega)$. Then, there exists a constant $\gamma>0$ such that
\begin{align}\label{eq:asy-limit-v}
\left\Vert v(t,\cdot) - \langle v^{\rm in} \rangle \right\Vert_{\mathrm H^1(\Omega)} \le e^{-\gamma t} \left\Vert v^{\rm in} \right\Vert_{\mathrm H^1(\Omega)}
\qquad \mbox{ for all }\, t>0
\end{align}
where $\langle v^{\rm in} \rangle$ denotes the initial average, i.e.,
\begin{align*}
\langle v^{\rm in} \rangle := \frac{1}{\left\vert \Omega \right\vert} \int_\Omega v^{\rm in}(x)\, {\rm d}x.
\end{align*}
\end{prop}
Proof of the above proposition is standard and is based in the spectral study of the corresponding elliptic problem, see e.g. \cite{Pazy_1983}. Alternatively, we could also use energy methods based on a priori estimates on the solution of the parabolic partial differential equations, see e.g. \cite[Chapter 13]{Chipot_2000}. As the proof of Proposition \ref{prop:v-H1-t-infty} is quite standard, we omit the proof.\\
Let us recall certain notions necessary for our proof. Let $\{\varphi_k\}_{k=1}^\infty\subset \mathrm H^1(\Omega)$ denote the collection of Neumann Laplacian eigenfunctions on $\Omega$, i.e., for each $k\in\mathbb{N}$,
\begin{equation}\label{eq:asy-spectral-pb-Neumann-Laplacian}
\begin{aligned}
-\Delta \varphi_k & = \mu_k \varphi_k \qquad \mbox{ in }\Omega,
\\[0.2 cm]
\nabla \varphi_k \cdot {\bf n}(x) & = 0 \qquad \quad \mbox{ on }\partial\Omega,
\end{aligned}
\end{equation}
with $\{\mu_k\}$ denoting the eigenvalues. Recall that the spectrum is discrete, non-negative and with no finite accumulation point, i.e.,
\begin{align*}
0=\mu_1 \le \mu_2 \le \cdots \to \infty.
\end{align*}

The next proposition is a result similar in flavour to Proposition \ref{prop:v-H1-t-infty}, but in a more general parabolic setting with high contrast in density and conductivity coefficients. More precisely, let us consider the heat equation with uniform conductivity in the presence of a defect with high contrast coefficients. We particularly choose the conductivity matrix to be \eqref{eq:A-small-inclusion} and the density coefficient to be \eqref{eq:rho-small-inclusion}. For an unknown $v^\eps(t,x)$, consider the initial-boundary value problem
\begin{equation}\label{eq:spec:small-inclusion}
\begin{aligned}
\rho^\eps(x) \partial_t v^\eps & = \nabla \cdot \Big( A^\eps(x) \nabla v^\eps \Big) \qquad \mbox{ in }(0,\infty)\times\Omega,
\\[0.2 cm]
\nabla v^\eps \cdot {\bf n}(x) & = 0 \quad \qquad\qquad\qquad \qquad\mbox{ on }(0,\infty)\times\partial\Omega,
\\[0.2 cm]
v^\eps(0,x) & = v^{\rm in}(x) \qquad \qquad\qquad \qquad\mbox{ in }\Omega.
\end{aligned}
\end{equation}
We give an asymptotic result for the solution $v^\eps(t,x)$ to \eqref{eq:spec:small-inclusion} in the $t\to\infty$ limit. As before, we are interested in the $\left\Vert \cdot \right\Vert_{\mathrm H^1}$-norm rather than the $\left\Vert \cdot \right\Vert_{\mathrm L^2}$-norm.

\begin{prop}\label{prop:v-eps-H1-t-infty}
Let $v^\eps(t,x)$ be the solution to the initial-boundary value problem \eqref{eq:spec:small-inclusion}. Suppose the initial datum $v^{\rm in}\in\mathrm H^1(\Omega)\cap \mathrm L^\infty(\Omega)$. Then, there exists a constant $\gamma_\eps>0$ such that for all $t>0$, we have
\begin{align}\label{eq:asy-limit-v-eps}
\left\Vert v^\eps(t,\cdot) - \mathfrak{m}^\eps \right\Vert_{\mathrm H^1(\Omega)} \le e^{-\gamma_\eps t} \left( \frac{1}{\eps^{d}} \left\Vert v^{\rm in} \right\Vert_{\mathrm L^2(\Omega)} + \frac{1}{\eps^{\frac{d-2}{2}}} \left\Vert \nabla v^{\rm in} \right\Vert_{\mathrm L^2(\Omega)} \right)
\end{align}
where $\mathfrak{m}^\eps$ denotes the following weighted initial average
\[
\mathfrak{m}^\eps := \frac{1}{\left\vert \Omega \right\vert \langle \rho^\eps \rangle} \int_\Omega \rho^\eps(x) v^{\rm in}(x)\, {\rm d}x
\qquad 
\mbox{ with }
\quad 
\langle \rho^\eps \rangle = \frac{1}{\left\vert \Omega \right\vert}\int_\Omega \rho^\eps(x)\, {\rm d}x.
\]
\end{prop}

\begin{rem}
From estimate \eqref{eq:asy-limit-v-eps} we conclude that the estimate on the right hand side is a product of an exponential decay term and a term of $\mathcal{O}(\eps^{-d})$. So, if $\gamma^\eps\gtrsim 1$, uniformly in $\eps$, then the solution converges to the weighted initial average in the long time regime. We will demonstrate that the decay rate $\gamma^\eps$ is bounded away from zero uniformly (with respect to $\eps$) in subsection \ref{ssec:schrodinger}.

If the density coefficient $\rho^\eps(x)$ is given by \eqref{eq:rho-small-inclusion}, then
\[
\langle \rho^\eps \rangle = \frac{1}{\left\vert \Omega \right\vert} \left\{ \int_{\Omega\setminus B_\eps} 1\, {\rm d}x + \frac{1}{\eps^d} \int_{B_\eps} 1\, {\rm d}x \right\} = \frac{1}{\left\vert \Omega \right\vert} \left( \left\vert \Omega\setminus B_\eps\right\vert + \frac{\pi^\frac{d}{2}}{\Gamma\left(\frac{d}{2} + 1\right)} \right).
\]
where $\Gamma(\cdot)$ denotes the gamma function. Substituting for $\langle \rho^\eps \rangle$ in the expression for the weighted initial average yields
\[
\mathfrak{m}^\eps =  \left( \left\vert \Omega\setminus B_\eps\right\vert + \frac{\pi^\frac{d}{2}}{\Gamma\left(\frac{d}{2} + 1\right)} \right)^{-1} \left\{ \int_{\Omega\setminus B_\eps} v^{\rm in}(x)\, {\rm d}x + \frac{1}{\eps^d} \int_{B_\eps} v^{\rm in}(x)\, {\rm d}x \right\}.
\]
Using the assumption on the initial datum $v^{\rm in}$ that it belongs to $\mathrm H^1(\Omega)\cap \mathrm L^\infty(\Omega)$, we get the following uniform bound (uniform with respect to $\eps$) on the weighted initial average
\[
\left\vert \mathfrak{m}^\eps \right\vert
\le 1 + \frac{\left\vert \Omega\right\vert \Gamma\left(\frac{d}{2} + 1\right)}{\pi^\frac{d}{2}} \left\Vert v^{\rm in} \right\Vert_{\mathrm L^2(\Omega)} < \infty.
\]
\end{rem}

\begin{proof}[Proof of Proposition \ref{prop:v-eps-H1-t-infty}]
Let $\mu^\eps_k$ and $\varphi^\eps_k$ be the Neumann eigenvalues and eigenfunctions defined as
\begin{equation}\label{eq:spec:eigenpair-small-inclusion}
\begin{aligned}
-\nabla \cdot \Big( A^\eps \nabla  \varphi^\eps_k \Big) & = \mu^\eps_k \rho^\eps \varphi^\eps_k \qquad \mbox{ in }\Omega,
\\[0.2 cm]
\nabla \varphi^\eps_k \cdot {\bf n}(x) & = 0 \qquad\qquad \quad \mbox{ on }\partial\Omega,
\end{aligned}
\end{equation}
where the conductivity-density pair $A^\eps(x)$, $\rho^\eps(x)$ is given by \eqref{eq:A-small-inclusion}-\eqref{eq:rho-small-inclusion}. Here again the spectrum is discrete, non-negative and with no finite accumulation point, i.e.,
\begin{align*}
0=\mu^\eps_1 \le \mu^\eps_2 \le \cdots \to \infty.
\end{align*}
The solution $v^\eps(t,x)$ to \eqref{eq:spec:small-inclusion} can be represented in terms of the basis functions $\{\varphi^\eps_k\}_{k=1}^\infty$ as
\begin{align*}
v^\eps(t,x) = \sum_{k=1}^\infty \mathfrak{b}^\eps_k(t) \varphi^\eps_k(x) 
\qquad
\mbox{ with }
\quad
\mathfrak{b}^\eps_k(t) = \int_\Omega v^\eps(t,x)\rho^\eps(x)\varphi^\eps_k(x)\, {\rm d}x.
\end{align*}
The general representation formula for the solution to \eqref{eq:spec:small-inclusion} becomes
\begin{align}\label{eq:spec:represent-v-eps}
v^\eps(t,x) = \sum_{k=1}^\infty \mathfrak{b}^\eps_k(0) e^{-\mu_kt} \varphi_k(x) = \mathfrak{b}^\eps_1(0)  \varphi^\eps_1 + \sum_{k=2}^\infty \mathfrak{b}^\eps_k(0) e^{-\mu^\eps_kt} \varphi^\eps_k(x).
\end{align}
The first term in the above representation is nothing but the weighted initial average
\begin{align*}
\varphi^\eps_1\mathfrak{b}^\eps_1(0) = \left\vert\varphi^\eps_1\right\vert^2 \int_\Omega v^{\rm in}(x) \rho^\eps(x) \, {\rm d}x = \frac{1}{\left\vert \Omega \right\vert \langle \rho^\eps\rangle} \int_\Omega v^{\rm in}(x) \rho^\eps(x) \, {\rm d}x =: \mathfrak{m}^\eps.
\end{align*}
The representation \eqref{eq:spec:represent-v-eps} says that the weighted $\mathrm L^2(\Omega)$-norm of $v^\eps(t,x)$ is 
\begin{align*}
\left\Vert v^\eps(t,\cdot) \right\Vert^2_{\mathrm L^2(\Omega;\rho^\eps)} = \left\vert \mathfrak{b}^\eps_1(0) \right\vert^2 + \sum_{k=2}^\infty \left\vert \mathfrak{b}^\eps_k(0) \right\vert^2 e^{-2\mu^\eps_kt}
\end{align*}
where we have used the following notation for the weighted Lebesgue space norm:
\[
\left\Vert w \right\Vert_{\mathrm L^2(\Omega;\rho^\eps)} := \int_\Omega w(x) \rho^\eps(x)\, {\rm d}x.
\]
The density coefficient $\rho^\eps(x)$ defined in \eqref{eq:rho-small-inclusion} appears as the weight function in the above Lebesgue space. It follows that
\[
\left\Vert w \right\Vert_{\mathrm L^2(\Omega)} \lesssim \left\Vert w \right\Vert_{\mathrm L^2(\Omega;\rho^\eps)} \lesssim \frac{1}{\eps^{d}} \left\Vert w \right\Vert_{\mathrm L^2(\Omega)}
\]
thanks to the definition of $\rho^\eps(x)$ in \eqref{eq:rho-small-inclusion}. Computing the weighted norm of the difference $v^\eps(t,x) - \mathfrak{m}^\eps$, we get
\begin{equation*}
\begin{aligned}
\left\Vert v^\eps(t,\cdot) - \mathfrak{m}^\eps \right\Vert^2_{\mathrm L^2(\Omega;\rho^\eps)} 
= \sum_{k=2}^\infty \left\vert \mathfrak{b}^\eps_k(0) \right\vert^2 e^{-2\mu^\eps_kt} 
\le e^{-2\mu^\eps_2t} \left\Vert v^{\rm in} \right\Vert^2_{\mathrm L^2(\Omega;\rho^\eps)}
\end{aligned}
\end{equation*}
Hence we deduce
\begin{align}\label{eq:v-eps-L2-t-infty}
\left\Vert v^\eps(t,\cdot) - \mathfrak{m}^\eps \right\Vert^2_{\mathrm L^2(\Omega)} 
\le 
e^{-2\mu^\eps_2t} \frac{1}{\eps^{2d}} \left\Vert v^{\rm in} \right\Vert^2_{\mathrm L^2(\Omega)}
\end{align}
Next, it follows from the representation formula \eqref{eq:spec:represent-v-eps} that
\begin{align*}
\sqrt{A^\eps(x)} \nabla v^\eps(t,x) = \sum_{k=2}^\infty \mathfrak{b}^\eps_k(0) e^{-\mu^\eps_kt} \sqrt{A^\eps(x)} \nabla \varphi^\eps_k(x)
\end{align*}
which in turn implies
\begin{align*}
\left\Vert \sqrt{A^\eps} \nabla v^\eps \right\Vert^2_{\mathrm L^2(\Omega)}
& = \sum_{k=2}^\infty \left\vert \mathfrak{b}^\eps_k(0) \right\vert^2 \mu^\eps_k\, e^{-2\mu^\eps_kt}
\le e^{-2\mu^\eps_2t} \sum_{k=2}^\infty \left\vert \mathfrak{b}^\eps_k(0) \right\vert^2 \mu^\eps_k
\\[0.2 cm]
& 
= e^{-2\mu^\eps_2t} \left\Vert \sqrt{A^\eps} \nabla v^{\rm in} \right\Vert^2_{\mathrm L^2(\Omega)}
\le e^{-2\mu^\eps_2t}\frac{1}{\eps^{d-2}} \left\Vert \nabla v^{\rm in} \right\Vert^2_{\mathrm L^2(\Omega)}
\end{align*}
Furthermore, as $1 \lesssim \left\Vert A^\eps \right\Vert_{\mathrm L^\infty}$, it follows that
\begin{align*}
\left\Vert \nabla v^\eps \right\Vert^2_{\mathrm L^2(\Omega)}
\le e^{-2\mu^\eps_2t} \frac{1}{\eps^{d-2}} \left\Vert \nabla v^{\rm in} \right\Vert^2_{\mathrm L^2(\Omega)}
\end{align*}
Repeating the above computations for the difference $v^\eps(t,x) - \mathfrak{m}^\eps$, we obtain
\begin{align}\label{eq:nabla-v-eps-L2-t-infty}
\left\Vert \nabla \left( v^\eps - \mathfrak{m}^\eps \right) \right\Vert^2_{\mathrm L^2(\Omega)}
\le e^{-2\mu^\eps_2t} \frac{1}{\eps^{d-2}} \left\Vert \nabla v^{\rm in} \right\Vert^2_{\mathrm L^2(\Omega)}
\end{align}
Gathering the inequalities \eqref{eq:v-eps-L2-t-infty}-\eqref{eq:nabla-v-eps-L2-t-infty} together proves the proposition with the constant $\gamma_\eps = \mu^\eps_2$, i.e., the first non-zero Neumann eigenvalue in the spectral problem \eqref{eq:spec:eigenpair-small-inclusion} on $\Omega$ with the high contrast density-conductivity pair $(\rho^\eps(x), A^\eps(x))$.
\end{proof}

\subsection{Reduction to a Schr\"odinger operator}\label{ssec:schrodinger}
The constants $\gamma$ and $\gamma^\eps$ in Propositions
\ref{prop:v-H1-t-infty} and \ref{prop:v-eps-H1-t-infty} respectively
give the rate of convergence to equilibrium. As the proofs in the previous subsection suggest, these rates are nothing but the first non-zero eigenvalues of the associated Neumann eigenvalue problems.

The constant decay rate $\gamma_\eps$ in Proposition \ref{prop:v-eps-H1-t-infty} may depend on the regularisation parameter $\eps$. In our setting, $\eps$ is nothing but the radius of the inclusion. Hence, to understand the behaviour of $\gamma^\eps$ in terms of $\eps$, we need to study the perturbations in the eigenvalues caused by the presence of inhomogeneities with conductivities and densities different from the background conductivity and density. More importantly, we need to understand the spectrum of transmission problems with high contrast conductivities and densities. The spectral analysis of elliptic operators with such conductivity matrices is done extensively in the literature in the context of electric impedance tomography -- see \cite{Ammari_2003, Ammari_2009} and references therein for further details. In \cite{Ammari_2003}, for example, the authors give an asymptotic expansion for the eigenvalues $\mu^\eps_k$ in terms of the regularising parameter $\eps$ (even in the case of multiplicities). We refer the readers to \cite[Eq.(23), page 74]{Ammari_2003} for the precise expansion. For high contrast conductivities such as $A^\eps$, refer to the concluding remarks in \cite[pages 74-75]{Ammari_2003} and references therein. Another important point to be noted is that the above mentioned works of Ammari and co-authors do not treat high contrast densities while addressing the spectral problems. For our setting -- more specifically for the spectral problem \eqref{eq:spec:eigenpair-small-inclusion} -- the readers are directed to consult the review paper of Chechkin \cite{Chechkin_2006} which goes in detail about the spectral problem in a related setting. Furthermore, the review paper \cite{Chechkin_2006} gives exhaustive reference to literature where similar spectral problems are addressed.

Rather than deducing the behaviour of the first non-zero eigenvalue of the spectral problem \eqref{eq:spec:eigenpair-small-inclusion} from \cite{Chechkin_2006}, we propose an alternate approach. Studying \eqref{eq:spec:eigenpair-small-inclusion} is the same as the study of the spectral problem for the operator
\begin{align}\label{eq:schrod:operator-L}
\mathcal{L}\, h := \frac{1}{\rho^\eps(x)} \nabla \cdot \Big( A^\eps(x) \nabla  h(x) \Big).
\end{align}
The idea is to show that the study of the spectral problem for $\mathcal{L}$ is analogous to the study of the spectral problem for a Schr\"odinger-type operator. The essential calculations to follow are inspired by the calculations in \cite[section 4.9, page 125]{Pavliotis_2014}. Note that the operator $\mathcal{L}$ defined by \eqref{eq:schrod:operator-L} with zero Neumann boundary condition is a symmetric operator in $\mathrm L^2(\Omega;\rho^\eps)$, i.e.,
\[
\int_{\R^d} \mathcal{L}\, h_1(x) h_2(x) \rho^\eps(x)\, {\rm d}x = \int_{\R^d} \mathcal{L}\, h_2(x) h_1(x) \rho^\eps(x)\, {\rm d}x
\]
for all $h_1, h_2\in\mathrm L^2(\R^d;\rho^\eps)$.\\
Let us now define the operator 
\begin{align}\label{eq:schrod:operator-H}
\mathcal{H}\, h := \sqrt{\rho^\eps(x)} \, \mathcal{L}\left(\frac{h}{\sqrt{\rho^\eps(x)}} \right) = \frac{1}{\sqrt{\rho^\eps(x)}} \nabla\cdot \left( \rho^\eps(x) \Sigma^\eps(x) \nabla \left(\frac{h}{\sqrt{\rho^\eps(x)}}\right)  \right),
\end{align}
with the coefficient 
\[
\Sigma^\eps(x) := \frac{A^\eps(x)}{\rho^\eps(x)}.
\]
An algebraic manipulation yields
\begin{align*}
\mathcal{H}\, h = \nabla\cdot \Big( \Sigma^\eps(x) \nabla h \Big) + W^\eps(x)\, h
\end{align*}
with
\begin{align*}
W^\eps(x) := \frac{1}{\sqrt{\rho^\eps(x)}} \nabla \cdot \left( A^\eps(x) \nabla \left( \frac{1}{\sqrt{\rho^\eps(x)}} \right) \right).
\end{align*}
In our setting, with the high contrast coefficients $A^\eps$ and $\rho^\eps$ from \eqref{eq:A-small-inclusion}-\eqref{eq:rho-small-inclusion}, the coefficients $\Sigma^\eps$ and $W^\eps$ become
\begin{equation*}
\Sigma^\eps(x) =
\left\{
\begin{array}{ll}
{\rm Id} & \quad \mbox{ for }x\in\Omega\setminus B_\eps,\\[0.2 cm]
\eps^2\, \frac{\beta}{\eta}\left(\frac{x}{\eps}\right) & \quad \mbox{ for }x\in B_\eps.
\end{array}\right.
\end{equation*}
and
\begin{equation*}
W^\eps(x) =
\left\{
\begin{array}{ll}
0 & \quad \mbox{ for }x\in\Omega\setminus B_\eps,\\[0.2 cm]
\eps^2\, \nabla \cdot \left(\beta\left(\frac{x}{\eps}\right) \nabla \left(\frac{1}{\sqrt{\eta}} \left(\frac{x}{\eps}\right) \right) \right)  & \quad \mbox{ for }x\in B_\eps.
\end{array}\right.
\end{equation*}
By definition \eqref{eq:schrod:operator-H}, the operators $\mathcal{L}$ and $\mathcal{H}$ are unitarily equivalent. Hence, they have the same eigenvalues. Note that the operator $\mathcal{H}$ is a Schr\"odinger-type operator where the coefficients are of high contrast.\\
By the Rayleigh-Ritz criterion, we have the characterisation
\begin{equation*}
\displaystyle \mu_k^\eps \le \frac{\int_\Omega \Sigma^\eps(x) \nabla \varphi(x)\cdot \nabla \varphi(x)\, {\rm d}x + \int_\Omega W^\eps(x) \left\vert \varphi(x) \right\vert^2\, {\rm d}x}{\int_\Omega \left\vert \varphi(x) \right\vert^2\, {\rm d}x}
\end{equation*}
where $\varphi\not\equiv0$ and is orthogonal to first $k-1$ Neumann eigenfunctions $\left\{ \varphi^\eps_1, \dots, \varphi^\eps_{k-1}\right\}$. Note, in particular, for $k=2$ (i.e. the first non-zero eigenvalue) we have
\begin{align}\label{eq:spec:Rayleight-Ritz-high-contrast}
\left\Vert \varphi \right\Vert^2_{\mathrm L^2(\Omega)}\, \mu_2^\eps \le \int_{\Omega} \Sigma^\eps(x)\nabla \varphi(x) \cdot \nabla \varphi(x) \, {\rm d}x + \int_{B_\eps} W^\eps(x)\varphi(x) \, {\rm d}x
\end{align}
with $\varphi\in\mathrm H^1(\Omega)$ such that
\[
\int_\Omega \varphi(x)\, {\rm d}x = 0.
\]
Note that we have used the fact that $W^\eps(x)$ is supported on $B_\eps$. Note further that
\begin{align}\label{eq:spec:W-eps-O1}
W^\eps(x) = \eps^2 \left\{ \frac{1}{\eps^2} \beta\left(\frac{x}{\eps}\right) : \left[\nabla^2\left(\frac{1}{\sqrt{\eta}}\right)\right]\left(\frac{x}{\eps}\right) + \frac{1}{\eps^2} \left[\nabla\beta \right]\left(\frac{x}{\eps}\right) \cdot \left[\nabla\left(\frac{1}{\sqrt{\eta}}\right)\right]\left(\frac{x}{\eps}\right)\right\} = \mathcal{O}(1)
\end{align}
if we assume $\beta\in\mathrm W^{1,\infty}(B_1;\R^{d\times d})$ and $\eta\in\mathrm W^{2,\infty}(B_1)$.\\
Note further that the spectral problem \eqref{eq:asy-spectral-pb-Neumann-Laplacian} comes with the following characterisation of the first non-zero eigenvalue (again by the Rayleigh-Ritz criterion)
\begin{align}\label{eq:spec:Rayleight-Ritz}
\left\Vert \varphi \right\Vert^2_{\mathrm L^2(\Omega)}\, \mu_2 \le \int_{\Omega} \nabla \varphi(x) \cdot \nabla \varphi(x) \, {\rm d}x
\end{align}
with $\varphi\in\mathrm H^1(\Omega)$ such that
\[
\int_\Omega \varphi(x)\, {\rm d}x = 0.
\]
Subtracting \eqref{eq:spec:Rayleight-Ritz-high-contrast} from \eqref{eq:spec:Rayleight-Ritz} yields
\begin{align*}
\left\Vert \varphi \right\Vert^2_{\mathrm L^2(\Omega)}\, \left( \mu_2 - \mu^\eps_2 \right) \le \int_{B_\eps} \left( {\rm Id} - \eps^2 \frac{1}{\eta} \left(\frac{x}{\eps}\right) \beta \left(\frac{x}{\eps}\right) \right)\nabla \varphi(x) \cdot \nabla\varphi(x) \, {\rm d}x + \int_{B_\eps} W^\eps(x)\varphi(x) \, {\rm d}x
\end{align*}
Let us now take the test function $\varphi$ to be the normalised eigenfunction $\varphi_2(x)$ associated with the first non-zero eigenvalue $\mu_2$ for the Neumann Laplacian:
\begin{align*}
\left\vert \mu_2 - \mu^\eps_2 \right\vert \le \left\vert \int_{B_\eps} \left( {\rm Id} - \eps^2 \frac{1}{\eta} \left(\frac{x}{\eps}\right) \beta \left(\frac{x}{\eps}\right) \right)\nabla \varphi_2(x) \cdot \nabla\varphi_2(x) \, {\rm d}x + \int_{B_\eps} W^\eps(x)\varphi_2(x) \, {\rm d}x\right\vert
\end{align*}
Using the observation \eqref{eq:spec:W-eps-O1} that the potential $W^\eps$ is of $\mathcal{O}(1)$ and that the Neumann eigenfunctions are bounded in $\mathrm W^{p,\infty}$ for any $p<\infty$ \cite{Grieser_2002}, we have proved that
\[
\left\vert \mu_2 - \mu^\eps_2 \right\vert \lesssim \eps^d.
\]
In this subsection, we have essentially proved the following result.
\begin{prop}\label{prop:first-non-zero-eigenvalue}
Suppose that the high contrast conductivity-density pair $A^\eps, \rho^\eps$ is given by \eqref{eq:A-small-inclusion}-\eqref{eq:rho-small-inclusion} with $\beta\in\mathrm W^{1,\infty}(B_1;\R^{d\times d})$ and $\eta\in\mathrm W^{2,\infty}(B_1)$. Let $\mu^\eps_2$ and $\mu_2$ be the first non-zero eigenvalues associated with the Neumann spectral problems \eqref{eq:spec:eigenpair-small-inclusion} and \eqref{eq:asy-spectral-pb-Neumann-Laplacian} respectively. Then we have
\begin{align}\label{eq:prop-eigenvalues-close}
\left\vert \mu_2 - \mu^\eps_2 \right\vert \lesssim \eps^d.
\end{align}
\end{prop}
Hence, as a corollary to the above result, we can deduce that the decay rate for the homogeneous transient problem \eqref{eq:spec-ibvp-v} and that for the high contrast transient problem \eqref{eq:spec:small-inclusion} are close to each other in the $\eps\ll1$ regime.

\section{Near-cloaking result}\label{sec:near-cloak-proof}
In this section, we will prove the main result of this paper, i.e., Theorem \ref{thm:near-cloak}. To that end, we first consider the steady-state problem associated with the homogeneous heat equation \eqref{eq:intro-u_hom}. More precisely, for the unknown $u_{\rm hom}^{\rm eq}(x)$, consider
\begin{equation}\label{eq:ibvp-u_hom-equil}
\begin{aligned}
-\Delta u_{\rm hom}^{\rm eq}(x) & = f(x) \qquad \mbox{ in }\Omega,
\\[0.2 cm]
\nabla u_{\rm hom}^{\rm eq} \cdot {\bf n}(x) & = g(x) \qquad\mbox{ on }\partial\Omega,
\\[0.2 cm]
\int_\Omega u_{\rm hom}^{\rm eq}(x)\, {\rm d}x & = 0.
\end{aligned}
\end{equation}
Note that the last line of \eqref{eq:ibvp-u_hom-equil} is to ensure that we solve for $u_{\rm hom}^{\rm eq}(x)$ uniquely. Here we assume that the source terms are admissible in the sense of \eqref{eq:thm-near-cloak-compatible}. This guarantees the solvability of the above elliptic boundary value problem.\\
Next we record a corollary to Proposition \ref{prop:v-H1-t-infty} which says how quickly the solution $u_{\rm hom}(t,x)$ to the homogeneous problem \eqref{eq:intro-u_hom} tends to its equilibrium state.
\begin{cor}\label{cor:uhom-time-decay}
Let $u_{\rm hom}(t,x)$ be the solution to \eqref{eq:intro-u_hom} and let $u_{\rm hom}^{\rm eq}(x)$ be the solution to the steady-state problem \eqref{eq:ibvp-u_hom-equil}. Suppose the source terms $f(x)$, $g(x)$ and the initial datum $u^{\rm in}(x)$ are admissible in the sense of \eqref{eq:thm-near-cloak-compatible}-\eqref{eq:thm-near-cloak-compatible-initial}. Then
\begin{align}\label{eq:cor:u-hom-asy-limit}
\left\Vert u_{\rm hom}(t,\cdot) - u_{\rm hom}^{\rm eq}(\cdot) \right\Vert_{\mathrm H^1(\Omega)} \le e^{-\gamma t} \left\Vert u^{\rm in} - u_{\rm hom}^{\rm eq} \right\Vert_{\mathrm H^1(\Omega)}
\end{align}
for some positive constant $\gamma$.
\end{cor}
\begin{proof}
Define a function $w(t,x):= u_{\rm hom}(t,x) - u_{\rm hom}^{\rm eq}(x)$. We have that the function $w(t,x)$ satisfies the evolution equation
\begin{equation*}
\begin{aligned}
\partial_t w & = \Delta w \qquad \qquad \qquad\qquad \mbox{ in }(0,\infty)\times\Omega,
\\[0.2 cm]
\nabla w \cdot {\bf n}(x) & = 0 \quad \qquad \qquad \qquad\qquad \mbox{ on }(0,\infty)\times\partial\Omega,
\\[0.2 cm]
w(0,x) & = u^{\rm in}(x) - u_{\rm hom}^{\rm eq}(x) \qquad\qquad \mbox{ in }\Omega,
\end{aligned}
\end{equation*}
which is the same as \eqref{eq:spec-ibvp-v}. The estimate \eqref{eq:cor:u-hom-asy-limit} is simply deduced from Proposition \ref{prop:v-H1-t-infty} (see in particular \eqref{eq:asy-limit-v}).
\end{proof}

Now we record a result, as a corollary to Proposition \ref{prop:v-eps-H1-t-infty}, demonstrating how quickly the solution $u^\eps(t,x)$ to the defect problem with high contrast coefficients \eqref{eq:intro:small-inclusion} tends to its equilibrium state. Consider the steady-state problem associated with the defect problem \eqref{eq:intro:small-inclusion}. More precisely, for the unknown $u^\eps_{\rm eq}(x)$, consider the elliptic boundary value problem
\begin{equation}\label{eq:intro-u_eps-equil}
\begin{aligned}
-\nabla \cdot \Big( A^\eps(x) \nabla u^\eps_{\rm eq} \Big)  & = f(x) \qquad \mbox{ in }\Omega,
\\[0.2 cm]
\nabla u^\eps_{\rm eq} \cdot {\bf n}(x) & = g(x) \qquad\mbox{ on }\partial\Omega,
\\[0.2 cm]
\int_\Omega \rho^\eps(x) u^\eps_{\rm eq}(x)\, {\rm d}x & = 0.
\end{aligned}
\end{equation}
Note that the normalisation condition in \eqref{eq:intro-u_eps-equil} makes use of the density coefficient $\rho^\eps(x)$ defined by \eqref{eq:rho-small-inclusion}.
\begin{cor}\label{cor:ueps-time-decay}
Let $u^\eps(t,x)$ be the solution to \eqref{eq:intro:small-inclusion} and let $u^\eps_{\rm eq}(x)$ be the solution to the steady-state problem \eqref{eq:intro-u_eps-equil}. Suppose the source terms $f(x)$, $g(x)$ and the initial datum $u^{\rm in}(x)$ are admissible in the sense of \eqref{eq:thm-near-cloak-compatible}-\eqref{eq:thm-near-cloak-compatible-initial}. Suppose further that ${\rm supp}\, u^{\rm in}\subset \Omega \setminus B_2$. Then
\begin{equation}
\begin{aligned}\label{eq:cor:u-eps-asy-limit}
\left\Vert u^\eps(t,\cdot) - u^\eps_{\rm eq}(\cdot) \right\Vert_{\mathrm H^1(\Omega)} &
\\
&  \le e^{-\gamma_\eps t} \left( \frac{1}{\eps^{d}} \left\Vert u^{\rm in} - u^\eps_{\rm eq} \right\Vert_{\mathrm L^2(\Omega)} + \frac{1}{\eps^{\frac{d-2}{2}}} \left\Vert \nabla \left( u^{\rm in} - u^\eps_{\rm eq} \right) \right\Vert_{\mathrm L^2(\Omega)} \right).
\end{aligned}
\end{equation}
\end{cor}

\begin{proof}
Define a function $w^\eps(t,x):= u^\eps(t,x) - u^\eps_{\rm eq}(x)$. We have that the function $w^\eps(t,x)$ satisfies the evolution equation
\begin{equation*}
\begin{aligned}
\rho^\eps(x)\partial_t w^\eps & = \nabla \cdot \Big( A^\eps(x) \nabla w^\eps \Big) \qquad \mbox{ in }(0,\infty)\times\Omega
\\[0.2 cm]
\nabla w^\eps \cdot {\bf n}(x) & = 0 \quad \qquad\qquad\qquad \qquad\mbox{ on }(0,\infty)\times\partial\Omega
\\[0.2 cm]
w^\eps(0,x) & = u^{\rm in}(x) - u^\eps_{\rm eq}(x) \qquad \qquad\qquad \qquad\mbox{ in }\Omega
\end{aligned}
\end{equation*}
which is the same as \eqref{eq:spec:small-inclusion}. By assumption, the initial datum $u^{\rm in}$ is supported away from $B_2$. This along with the admissibility assumption \eqref{eq:thm-near-cloak-compatible-initial} implies that the weighted average $\mathfrak{m}^\eps$ defined in Proposition \ref{prop:v-eps-H1-t-infty} vanishes. Then the estimate \eqref{eq:cor:u-eps-asy-limit} is simply deduced from Proposition \ref{prop:v-eps-H1-t-infty} (see in particular \eqref{eq:asy-limit-v-eps}).
\end{proof}

\begin{proof}[Proof of Theorem \ref{thm:near-cloak}]
From the triangle inequality, we have
\begin{align*}
\left\Vert u^\eps(t,\cdot) - u_{\rm hom}(t,\cdot) \right\Vert_{\mathrm H^{\frac12}(\partial\Omega)}
& \le \left\Vert u^\eps(t,\cdot) - u^\eps_{\rm eq}(\cdot) \right\Vert_{\mathrm H^{\frac12}(\partial\Omega)}
+ \left\Vert u^\eps_{\rm eq} - u^{\rm eq}_{\rm hom} \right\Vert_{\mathrm H^{\frac12}(\partial\Omega)}
\\[0.2 cm]
& \quad + \left\Vert u^{\rm eq}_{\rm hom}(\cdot) - u_{\rm hom}(t,\cdot) \right\Vert_{\mathrm H^{\frac12}(\partial\Omega)}.
\end{align*}
The boundary trace inequality gives the existence of a constant $\mathfrak{c}_{\Omega}$ -- depending only on the domain $\Omega$ -- such that
\begin{align*}
\left\Vert u^\eps(t,\cdot) - u_{\rm hom}(t,\cdot) \right\Vert_{\mathrm H^{\frac12}(\partial\Omega)}
& \le \mathfrak{c}_\Omega \left\Vert u^\eps(t,\cdot) - u^\eps_{\rm eq}(\cdot) \right\Vert_{\mathrm H^1(\Omega)}
+ \left\Vert u^\eps_{\rm eq} - u^{\rm eq}_{\rm hom} \right\Vert_{\mathrm H^{\frac12}(\partial\Omega)}
\\[0.2 cm]
& \quad + \mathfrak{c}_\Omega \left\Vert u^{\rm eq}_{\rm hom}(\cdot) - u_{\rm hom}(t,\cdot) \right\Vert_{\mathrm H^1(\Omega)}.
\end{align*}
Using the result of Corollary \ref{cor:uhom-time-decay}, Corollary \ref{cor:ueps-time-decay} and the DtN estimate from \cite[Lemma 2.2, page 305]{Friedman_1989} (see also \cite[Proposition 1]{Kohn_2008}), we get that the right hand side of the above inequality is bounded from above by
\begin{align*}
& \mathfrak{c}_\Omega 
e^{-\gamma_\eps t} \left( \frac{1}{\eps^{d}} \left\Vert u^{\rm in} - u^\eps_{\rm eq} \right\Vert_{\mathrm L^2(\Omega)} + \frac{1}{\eps^{\frac{d-2}{2}}} \left\Vert \nabla \left( u^{\rm in} - u^\eps_{\rm eq} \right) \right\Vert_{\mathrm L^2(\Omega)} \right)
\\[0.2 cm]
&+ \eps^d \left( \left\Vert f\right\Vert_{\mathrm L^2(\Omega)} +  \left\Vert g\right\Vert_{\mathrm L^2(\partial\Omega)} \right)
+ \mathfrak{c}_\Omega e^{-\gamma t} \left\Vert u^{\rm in} - u_{\rm hom}^{\rm eq} \right\Vert_{\mathrm H^1(\Omega)}.
\end{align*}
Hence the existence of a time instant $\mathit{T}$ follows such that for all $t\ge\mathit{T}$, we indeed have the estimate \eqref{eq:thm:H12-estimate}.
\end{proof}

\begin{rem}\label{rem:near-cloak-m-eps-avg}
We make some observations on why the admissibility assumption \eqref{eq:thm-near-cloak-compatible-initial} and the assumption on the support of the initial datum in Corollary \ref{cor:ueps-time-decay} were essential to our proof. In the absence of these assumptions, in the proof of Theorem \ref{thm:near-cloak}, we will have to show that
\begin{align}\label{eq:near-cloak-remark-m-eps-avg}
\lim_{\eps\to0} \left\vert \mathfrak{m}^\eps - \langle u^{\rm in} \rangle \right\vert = 0.
\end{align}
Let us compute the difference
\begin{align*}
\mathfrak{m}^\eps - \langle u^{\rm in} \rangle 
& = \left( \left\vert \Omega\setminus B_\eps\right\vert + \frac{\pi^\frac{d}{2}}{\Gamma\left(\frac{d}{2} + 1\right)} \right)^{-1} \left\{ \int_{\Omega\setminus B_\eps} u^{\rm in}(x)\, {\rm d}x + \frac{1}{\eps^d} \int_{B_\eps} u^{\rm in}(x)\, {\rm d}x \right\}
\\
& \quad - \frac{1}{\left\vert \Omega \right\vert} \int_{\Omega} u^{\rm in}(x)\, {\rm d}x.
\end{align*}
We have not managed to characterise all initial data that guarantee the asymptote \eqref{eq:near-cloak-remark-m-eps-avg}. The difficulty of this task becomes apparent if you take initial data to be supported away from $B_1$, but not satisfying the zero mean assumption \eqref{eq:thm-near-cloak-compatible-initial}. Note that our assumption on the initial data in Corollary \ref{cor:ueps-time-decay} guarantees the above difference is always zero, irrespective of the value of the parameter $\eps$. 
\end{rem}

\begin{rem}\label{rem:near-cloak-full-space}
Our choice of the domain $\Omega$ containing $B_2$ is arbitrary and Corollary \ref{cor:diff-cloak-hom} asserts that for any such arbitrary choice, the distance between the solutions $u_{\rm hom}$ and $u_{\rm cl}$ (measured in the $\mathrm H^\frac12(\partial\Omega)$-norm) can be made as small as we wish provided we engineer appropriate cloaking coefficients (for instance via a homogenization approach) -- see \eqref{eq:rho-cloak-choice} and \eqref{eq:A-cloak-choice} -- in the annulus
$B_2\setminus B_1$. This is the notion of \emph{near-cloak}. Unlike the \emph{perfect cloaking} strategies which demand equality between $u_{\rm hom}$ and $u_{\rm cl}$ everywhere outside $B_2$, near-cloak strategies only ask for them to be close in certain norm topologies. Near-cloaking strategies are what matters in practice.\\
We can still pose the thermal cloaking question in full space $\R^d$. In this scenario, the norm of choice becomes the one in $\mathrm H^1_{\rm loc}(\R^d\setminus B_2)$. More precisely, for any compact set $K\subset \R^d\setminus B_2$, we can prove
\[
\left\Vert u^\eps(t,\cdot) - u_{\rm hom}(t,\cdot) \right\Vert_{\mathrm H^1(K)} \lesssim \eps^{\frac{d}{2}} \qquad \mbox{ for }t\gg 1.
\]
\end{rem}

\subsection{Layered cloaks}\label{ssec:layer-cloak}
Advancing an idea from \cite{Gralak_2016}, we develop a transformation media theory for thermal layered cloaks, which are of practical importance in thin-film solar cells for energy harvesting in the photovoltaic industry. The basic principle behind this construction is the following observation.
\begin{prop}\label{prop:layer-cloak}
Let the spatial domain $\Omega:=(-3,3)^2$. Let the density conductivity pair $\rho\in\mathrm L^\infty(\Omega;\R)$, $A\in\mathrm L^\infty(\Omega;\R^{2\times 2})$ be such that they are $(-3,3)$-periodic in the $x_1$ variable. Consider a smooth invertible map $\mathtt{f}(x_2):\R\mapsto\R$ such that $\mathtt{f}(x_2) = x_2$ for $\vert x_2\vert>2$. Assume further that $\mathtt{f}'(x_2)\ge C>0$ for a.e. $x_2\in(-3,3)$. Take the mapping $\mathbb{F}:\Omega\mapsto\Omega$ defined by $\mathbb{F}(x_1,x_2) = (x_1,\mathtt{f}(x_2))$. Then $u(t,x_1,x_2)$ is a $(-3,3)$-periodic solution (in the $x_1$ variable) to
\begin{align*}
\rho(x) \partial_t u = \nabla_x \cdot \Big( A(x) \nabla_x u \Big) + h(x)\qquad \mbox{ for }(t,x)\in(0,\infty)\times\Omega
\end{align*}
if and only if $v= u\circ \mathbb{F}^{-1}$ is a $(-3,3)$-periodic solution (in the $y_1=x_1$ variable) to
\begin{align*}
\mathbb{F}^*\rho(y)\, \partial_t v = \nabla \cdot \Big( \mathbb{F}^*A(y) \nabla v \Big) + \mathbb{F}^*f(y) \qquad \mbox{ for }(t,y)\in(0,\infty)\times\Omega
\end{align*}
where the coefficients are given below
\begin{equation}\label{eq:push-forward-layered-cloak-formulas}
\begin{aligned}
\mathbb{F}^*\rho(y_1,y_2) & = \frac{1}{\mathtt{f}'(x_2)}\rho(x_1,x_2);
\\[0.2 cm]
\mathbb{F}^*h(t,y_1,y_2) & = \frac{1}{\mathtt{f}'(x_2)} h(t,x_1,x_2);
\\[0.2 cm]
\mathbb{F}^*A(y_1,y_2) & = 
\left(
\begin{matrix}
\frac{1}{\mathtt{f}'(x_2)} A_{11}(x_1,x_2) & A_{12}(x_1,x_2)
\\[0.3 cm]
\frac{1}{\mathtt{f}'(x_2)} A_{21}(x_1,x_2) & \mathtt{f}'(x_2)A_{22}(x_1,x_2)
\end{matrix}
\right)
\end{aligned}
\end{equation}
with the understanding that the right hand sides in \eqref{eq:push-forward-layered-cloak-formulas} are computed at $(x_1,x_2)=(y_1,\mathtt{f}^{-1}(y_2))$. Furthermore, we have
\[
u(t,x_1,x_2) = v(t,x_1,x_2) \qquad \mbox{ for }\vert x_2\vert\ge2.
\] 
\end{prop}
Next, we prove a near-cloaking result in this present setting of layered cloaks. It concerns the following evolution problems: The homogeneous problem
\begin{equation}\label{eq:layer-cloak:u_hom}
\begin{aligned}
\partial_t u_{\rm hom} (t,x) & = \Delta u_{\rm hom}(t,x) + f(x) \qquad \mbox{ in }(0,\infty)\times\Omega,
\\[0.2 cm]
\nabla u_{\rm hom}(x_1,\pm3) \cdot {\bf n}(x_1,\pm3) & = g(x_1) \qquad \qquad\qquad \qquad\mbox{ on }(0,\infty)\times(-3,3),
\\[0.2 cm]
u_{\rm hom}(-3,x_2) & = u_{\rm hom}(3,x_2) \qquad \qquad\qquad \qquad\qquad \mbox{ for }x_2\in(-3,3),
\\[0.2 cm]
u_{\rm hom}(0,x) & = u^{\rm in}(x) \qquad \qquad\qquad \qquad\qquad \mbox{ in }\Omega
\end{aligned}
\end{equation}
and the layered cloak problem
\begin{equation}\label{eq:layer-cloak:u_cloak}
\begin{aligned}
\rho_{\rm cl}(x) \partial_t u_{\rm cl} & = \nabla \cdot \Big( A_{\rm cl}(x) \nabla u_{\rm cl} \Big) + f(x) \qquad \mbox{ in }(0,\infty)\times\Omega,
\\[0.2 cm]
\nabla u_{\rm cl}(x_1,\pm3) \cdot {\bf n}(x_1,\pm3) & = g(x_1) \qquad \qquad\qquad \qquad\mbox{ on }(0,\infty)\times(-3,3),
\\[0.2 cm]
u_{\rm cl}(-3,x_2) & = u_{\rm cl}(3,x_2) \qquad \qquad\qquad \qquad\qquad \mbox{ for }x_2\in(-3,3),
\\[0.2 cm]
u_{\rm cl}(0,x) & = u^{\rm in}(x) \qquad \qquad\qquad \qquad\qquad \mbox{ in }\Omega
\end{aligned}
\end{equation}
where the coefficients $\rho_{\rm cl}$ and $A_{\rm cl}$ in \eqref{eq:layer-cloak:u_cloak} are defined using the Lipschitz mapping $(x_1,x_2)\mapsto (x_1,\mathtt{f}_\eps(x_2))$ with 
\begin{equation}\label{eq:layered-cloak-map}
\mathtt{f}_\eps(x_2) :=
\left\{
\begin{array}{cl}
x_2 & \mbox{ for } \left\vert x_2 \right\vert>2
\\[0.2 cm]
\left( \frac{2-2\eps}{2-\eps} + \frac{\vert x_2 \vert}{2-\eps}\right) \frac{x_2}{\vert x_2\vert} & \mbox{ for } 1 \le \vert x_2\vert\le 2
\\[0.2 cm]
\frac{x_2}{\eps} & \mbox{ for }\vert x_2\vert <1.
\end{array}\right.
\end{equation}
The precise construction of the layered cloaks is as follows
\begin{equation}\label{eq:layer:rho-cloak-choice}
\rho_{\rm cl}(x_1,x_2) =
\left\{
\begin{array}{ll}
1 & \quad \mbox{ for }\vert x_2\vert >2,\\[0.2 cm]
\mathcal{F}^*_\eps 1 & \quad \mbox{ for }1<\vert x_2\vert <2,\\[0.2 cm]
\eta(x_1,x_2) & \quad \mbox{ for }\vert x_2 \vert <1
\end{array}\right.
\end{equation}
and
\begin{equation}\label{eq:layer:A-cloak-choice}
A_{\rm cl}(x_1,x_2) =
\left\{
\begin{array}{ll}
{\rm Id} & \quad \mbox{ for }\vert x_2\vert >2,\\[0.2 cm]
\mathcal{F}^*_\eps {\rm Id} & \quad \mbox{ for }1<\vert x_2\vert <2,\\[0.2 cm]
\beta(x_1,x_2) & \quad \mbox{ for }\vert x_2 \vert <1,
\end{array}\right.
\end{equation}
where the push-forward maps are defined in \eqref{eq:push-forward-layered-cloak-formulas}. The density coefficient $\eta(x)$ in \eqref{eq:layer:rho-cloak-choice} is any arbitrary real positive function. The conductivity coefficient $\beta(x)$ in \eqref{eq:layer:A-cloak-choice} is any arbitrary bounded positive definite matrix.

Let us make the observation that the cloaking coefficients $\rho_{\rm cl}$ and $A_{\rm cl}$ given by \eqref{eq:layer:rho-cloak-choice} and \eqref{eq:layer:A-cloak-choice} respectively can be treated as push-forward outcomes (via the push-forward maps \eqref{eq:push-forward-layered-cloak-formulas}) of the following defect coefficients
\begin{equation}\label{eq:layer:rho-small-inclusion}
\rho^\eps(x) =
\left\{
\begin{array}{ll}
1 & \quad \mbox{ for }\eps < \vert x_2\vert < 2,\\[0.2 cm]
\frac{1}{\eps}\eta\left(x_1, \frac{x_2}{\eps}\right) & \quad \mbox{ for }\vert x_2\vert < \eps
\end{array}\right.
\end{equation}
and
\begin{equation}\label{eq:layer:A-small-inclusion}
A^\eps(x) =
\left\{
\begin{array}{ll}
{\rm Id} & \quad \mbox{ for }\eps < \vert x_2\vert < 2,\\[0.2 cm]
\left(
\begin{matrix}
\frac{1}{\eps}\beta_{11}\left(x_1,\frac{x_2}{\eps}\right) & \beta_{12}\left(x_1,\frac{x_2}{\eps}\right)
\\[0.2 cm]
\beta_{21}\left(x_1,\frac{x_2}{\eps}\right) & \eps \beta_{22}\left(x_1,\frac{x_2}{\eps}\right)
\end{matrix}
\right) & \quad \mbox{ for }\vert x_2\vert < \eps.
\end{array}\right.
\end{equation}
It then follows from Proposition \ref{prop:layer-cloak} that comparing $u_{\rm hom}$ and $u_{\rm cl}$ is equivalent to comparing $u_{\rm hom}$ and $u^\eps$ where $u^\eps(t,x)$ solves the following defect problem with the aforementioned $\rho^\eps$ and $A^\eps$ as coefficients:
\begin{equation}\label{eq:layer-cloak:u_defect}
\begin{aligned}
\rho^\eps(x) \partial_t u^\eps & = \nabla \cdot \Big( A^\eps(x) \nabla u^\eps \Big) + f(x) \qquad \mbox{ in }(0,\infty)\times\Omega,
\\[0.2 cm]
\nabla u^\eps(x_1,\pm3) \cdot {\bf n}(x_1,\pm3) & = g(x_1) \qquad \qquad\qquad \qquad\mbox{ on }(0,\infty)\times(-3,3),
\\[0.2 cm]
u^\eps(-3,x_2) & = u_{\rm cl}(3,x_2) \qquad \qquad\qquad \qquad\qquad \mbox{ for }x_2\in(-3,3),
\\[0.2 cm]
u^\eps(0,x) & = u^{\rm in}(x) \qquad \qquad\qquad \qquad\qquad \mbox{ in }\Omega
\end{aligned}
\end{equation}
\begin{thm}\label{thm:one-d:near-cloak}
Let $u^\eps(t,x)$ be the solution to the defect problem \eqref{eq:layer-cloak:u_defect} with high contrast coefficients \eqref{eq:layer:rho-small-inclusion}-\eqref{eq:layer:A-small-inclusion} and let $u_{\rm hom}(t,x)$ be the solution to the homogeneous conductivity problem \eqref{eq:layer-cloak:u_hom}. Then, there exists a time instant $\mathit{T} <\infty$ such that for all $t\ge \mathit{T}$ we have
\begin{align}\label{eq:thm:layer:H12-norm}
\left\Vert u^\eps(t,\cdot) - u_{\rm hom}(t,\cdot) \right\Vert_{\mathrm H^\frac12(\Gamma)} \lesssim \eps^2.
\end{align}
\end{thm}
The proof is similar to the proof of Theorem \ref{thm:near-cloak}. More specifically, we show first that the solutions to the transient problems \eqref{eq:layer-cloak:u_defect} and \eqref{eq:layer-cloak:u_hom} converge exponentially fast to their corresponding equilibrium states. We can then adapt the energy approach in the proof of \cite{Friedman_1989} to show that the equilibrium states are $\eps^2$ close in the $\mathrm H^\frac12(\Gamma)$-norm. We note that our analysis of near-cloaking for thermal layered cloaks can be easily adapted to electrostatic and electromagnetic cases. It might also find interesting applications in Earth science for seismic tomography, in which case one could utilise results in \cite{Ammari_2013} to prove near cloaking results related to imaging the subsurface of the Earth with seismic waves produced by earthquakes.

\section{Numerical results}\label{sec:numerics}
This section deals with the numerical tests done in support of the theoretical results in the paper. The tests designed in the subsections to follow make some observations with regards to the near-cloaking scheme designed in the previous sections of this paper. The numerical simulations are done in one, two and three spatial dimensions. It seems natural to start with the one dimensional case, but as it turns out, see subsection \ref{ssec:layer-cloak}, the physical problem of interest for a one-dimensional (so-called layered) cloak requires a two-dimensonal computational domain, and thus we start with the 2D case. We refer the reader to \cite{Gralak_2016} for the precise physical setup and importance of the layered cloak described in the context of Maxwell's equations (this is easily translated into the language of conductivity equation). These numerical simulations were performed with the finite element software COMSOL MULTIPHYSICS.

We choose the spatial domain to be a square $\Omega:=(-3,3)^2$. We take the bulk source $f(x)$, the Neumann datum $g(x)$ and the initial datum $u^{\rm in}(x)$ to be smooth and such that ${\rm supp}\, f\subset \Omega\setminus B_2$, ${\rm supp}\, u^{\rm in}\subset \Omega\setminus B_2$. We further assume that
\[
\int_\Omega f(x)\, {\rm d}x = 0,
\qquad
\int_{\partial\Omega} g(x)\, {\rm d}\sigma(x) = 0,
\qquad
\int_\Omega u^{\rm in}(x)\, {\rm d}x = 0.
\]
This guarantees that the data is admissible in the sense of \eqref{eq:thm-near-cloak-compatible}-\eqref{eq:thm-near-cloak-compatible-initial}.
\subsection{Near-cloaking}\label{ssec:num:near-cloak}
The numerical experiments in this subsection analyse the sharpness of the near-cloaking result (Theorem \ref{thm:near-cloak}) in this paper. We solve the initial-boundary value problem for the unknown $u^\eps(t,x)$:
\begin{equation}\label{eq:num:defect-problem}
\begin{aligned}
\rho^\eps(x)\frac{\partial u^\eps}{\partial t} & = \nabla \cdot \Big( A^\eps(x) \nabla u^\eps \Big) + f(x) \qquad \mbox{ in }(0,\infty)\times\Omega,
\end{aligned}
\end{equation}
with the density-conductivity coefficients:
\begin{equation*}
\rho^\eps(x), \, A^\eps(x) =
\left\{
\begin{array}{ll}
1, \qquad \qquad \quad {\rm Id} & \quad \mbox{ for }x\in\Omega\setminus B_\eps,\\[0.2 cm]
\frac{1}{\eps^2}\eta\left(\frac{x}{\eps}\right), \qquad \beta\left(\frac{x}{\eps}\right) & \quad \mbox{ for }x\in B_\eps
\end{array}\right.
\end{equation*}
Note that in two-dimensions, there is no high contrast in the conductivity coefficient $A^\eps(x)$. There is, however, contrast in the density coefficient $\rho^\eps(x)$. The evolution \eqref{eq:num:defect-problem} is supplemented by the Neumann datum
\begin{equation*}
g(x_1,x_2) =
\left\{
\begin{array}{ll}
-3 & \quad \mbox{ for }x_1=\pm 3,\\[0.2 cm]
0 & \quad \mbox{ for }x_2=\pm 3
\end{array}\right.
\end{equation*}
and the initial datum
\begin{equation*}
u^{\rm in}(x_1,x_2) =
\left\{
\begin{array}{ll}
x_1 x_2 & \quad \mbox{ for }x\in\Omega\setminus B_2,\\[0.2 cm]
0 & \quad \mbox{ for }x\in B_\eps
\end{array}\right.
\end{equation*}
The bulk force in \eqref{eq:num:defect-problem} is taken to be
\begin{equation*}
f(x_1,x_2) =
\left\{
\begin{array}{ll}
\sqrt{x_1^2+x_2^2}\, \sin(x_1)\sin(x_2)-2 & \quad \mbox{ for }x\in\Omega\setminus B_2,\\[0.2 cm]
0 & \quad \mbox{ for }x\in B_\eps
\end{array}\right.
\end{equation*}
We next solve the initial-boundary value problem (Neumann) for $u_{\rm hom}(t,x)$ with the above data.
\begin{equation*}
\begin{aligned}
\partial_t u_{\rm hom} (t,x) & = \Delta u_{\rm hom}(t,x) + f(x) \qquad \mbox{ in }(0,\infty)\times\Omega.
\end{aligned}
\end{equation*}
Let us define
\begin{align}\label{eq:num:G-epsilon}
\mathcal{G}_\eps(t) := \frac{\left\Vert u^\eps(t,\cdot) - u_{\rm hom}(t,\cdot) \right\Vert_{\mathrm L^2(\partial\Omega)}}{\eps^d \left( \left\Vert u^{\rm in} \right\Vert_{\mathrm H^1} + \left\Vert f \right\Vert_{\mathrm L^2(\Omega)} + \left\Vert g \right\Vert_{\mathrm L^2(\partial\Omega)} \right)}
\end{align}
We compute the function $\mathcal{G}_\eps(t)$ defined by \eqref{eq:num:G-epsilon} as a function of time for various values of $\eps$, see Figure \ref{fig:num:G-epsilon}, and numerically observe that after $110s$, $\mathcal{G}_\eps(t)$ reaches an asymptote that tends towards a numerical value close to $8.6$ when $\eps$ gets smaller, in agreement with theoretical predictions of Theorem \ref{thm:near-cloak} for space dimension $d=2$. The simulations were performed using an adaptive mesh of the domain $\Omega$ consisting of $2\times10^5$ nodal elements (with at least $10^2$ elements in the small defect when $\eps=10^{-4}$), and a numerical solver based on the backward differential formula (BDF solver) with initial time step $10^{-5}s$, maximum time step $10^{-1}s$, minimum and maximum BDF orders of $1$ and $5$, respectively. Many test cases were ran for various values of $\beta$ (including anisotropic conductivity) and $\eta$ in the small inclusion, and we report some representative curves in Figure \ref{fig:num:G-epsilon}.

\begin{figure}[h!]
\includegraphics[width=10cm]{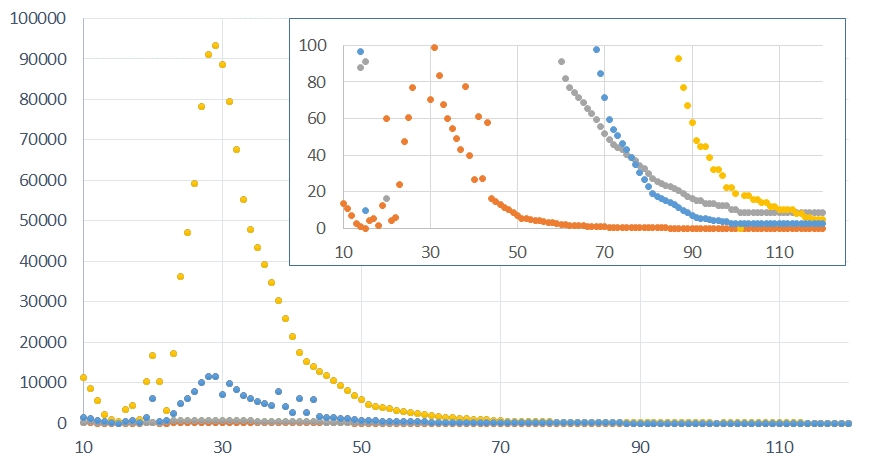}
\caption{Numerical results for ${\mathcal G}_\varepsilon(t)$ vs time $t$ (with a source outside $B_2$): Same parameters as in Figure \ref{fig:contour:2D} for a small defect of radius $\varepsilon=10^{-1}$ (orange), $\varepsilon=10^{-2}$ (grey), $\varepsilon=10^{-3}$ (blue), $\varepsilon=10^{-4}$ (yellow) and diffusivity $\beta(x/\varepsilon)=2$ Id and density $\varepsilon^{-2}\eta(x/\varepsilon)=2\varepsilon^{-2}$ in the inclusion $B_\eps$. Horizontal linear scale of time $t\in [0,110]$ s. Vertical scale of
$\varepsilon^{-2}\vert{\Vert u^\varepsilon(t,.)\Vert}_{L^2(\partial\Omega)}-{\Vert u_{hom}(t,.)\Vert}_{L^2(\partial\Omega)}\vert$ is a representation of ${\mathcal G}_\varepsilon(t)$. Insert shows a zoom-in.}
\label{fig:num:G-epsilon}
\end{figure}

\subsection{Contour plots}
Let us start with the contour plots for the layered cloak developed in subsection \ref{ssec:layer-cloak}, see Figure \ref{fig:layered-cloak}. More precisely, we compare solutions $u_{\rm hom}$ and $u_{\rm cl}$ to the evolution problems \eqref{eq:layer-cloak:u_hom} and \eqref{eq:layer-cloak:u_cloak} respectively, where periodic boundary conditions are imposed at the boundary $x_1=\pm3$ and homogeneous Neumann datum at the boundary $x_2=\pm3$. Here, we illustrate the layered cloak by numerically solving an equation for the unknown $u(t,x_1,x_2)$, but the conductivity and density only depend upon the $x_2$ variable. We take the source
\begin{equation*}
f(x_1,x_2) =
\left\{
\begin{array}{ll}
x_2\, \sin(x_2) & \quad \mbox{ for }x_2\in(-3,-2)\cup(2,3),\\[0.2 cm]
0 & \quad \mbox{ for }x_2\in (-2,2)
\end{array}\right.
\end{equation*}
and the initial datum is taken to be
\begin{equation*}
u^{\rm in}(x_1,x_2) =
\left\{
\begin{array}{ll}
x_2 & \quad \mbox{ for }x_2\in(-3,-2)\cup(2,3),\\[0.2 cm]
0 & \quad \mbox{ for }x_2\in (-2,2)
\end{array}\right.
\end{equation*}
Following the layered cloak construction in \eqref{eq:layer:A-cloak-choice}, we take the cloaking conductivity to be the push-forward of identity in the cloaking strip:
\[
A_{\rm cl}(x_1,x_2) = {\rm diag} \left(2-\eps, \frac{1}{2-\eps}\right) \qquad \mbox{ for }1< \vert x_2\vert <2.
\]
We report in Figure \ref{fig:layered-cloak} some numerical results that exemplify the high level of control of the heat flux with a layered cloak: in the upper panels (a), (b), (c), one can see snapshots at representative time steps (t=0, 1 and 4s) of a typical one-dimensional diffusion process in a homogeneous medium for a given source with a support outside $x_2\in(-2,2)$. When we compare the temperature field at initial time step in (a) with that when we replace homogeneous medium by a layered cloak in $1< \vert x_2\vert <2$ in (d), we note no difference. However, some noticeable differences are noted for the temperature field between the homogeneous medium and the cloak when comparing (b) with (d) and (c) with (e). The gradient of temperature field is dramatically decreased in the invisibility region $x_2\in(-1,1)$, leading to an almost uniform (but non zero) temperature field therein, and this is compensated by an enhanced gradient of temperature within the cloak in $1< \vert x_2\vert <2$. One notes that the increased flux in $1< \vert x_2\vert <2$ might be useful to improve efficiency of solar cells in photovoltaics.

Our next experiment is to compare the homogeneous solution and the cloaked solution where the cloaking coefficients are constructed using the push-forward maps as in \eqref{eq:rho-cloak-choice}-\eqref{eq:A-cloak-choice}. The data ($f$, $g$, $u^{\rm in}$) are chosen as in subsection \ref{ssec:num:near-cloak}. We report in Figure \ref{fig:contour:2D} some numerical computations performed in COMSOL that illustrate the strong similarity between the temperature fields in homogeneous (a-d), small defect (e-h) and cloaked (i-l) problems. Obviously, the fields are identical outside $B_2$ at initial time step, then they differ most outside $B_2$ at small time step t=1s, see (b),(f),(j), and become more and more similar with increasing time steps, see (c), (g), (k) and (d), (h), (l). These qualitative observations are consistent with the near-cloaking result, Theorem \ref{thm:near-cloak}, of this paper.

We have also performed a similar experiment in three dimensions, see Figure \ref{fig:contour:3D}. Refer to the caption in Figure \ref{fig:contour:3D} for the parameters considered in the three dimensional problem. Note that for these 3D computations, we mesh the cubical domain with $40,000$ nodal elements, we take time steps of $0.1s$ and use BDF solver with initial time step $0.01s$, maximum time step $0.1s$, minimum and maximum BDF orders of $1$ and $3$, respectively. We use a desktop with 32 Gb of RAM memory and a computation run takes around 1 hour for a time interval between $0.1$ and $10s$. We are not able to study long time behaviours, nor solve the high contrast small defect problem in 3D, as this would require more computational resources. Nevertheless, our 3D computations suggest that there is a strong similarity between the temperature fields in homogeneous and cloaked problems, outside $B_2$. Note also that we consider a source not vanishing inside $B_2$, which motivates further theoretical analysis for sources with a support in the overall domain $\Omega$.  

\begin{figure}[h!]
\includegraphics[width=10cm]{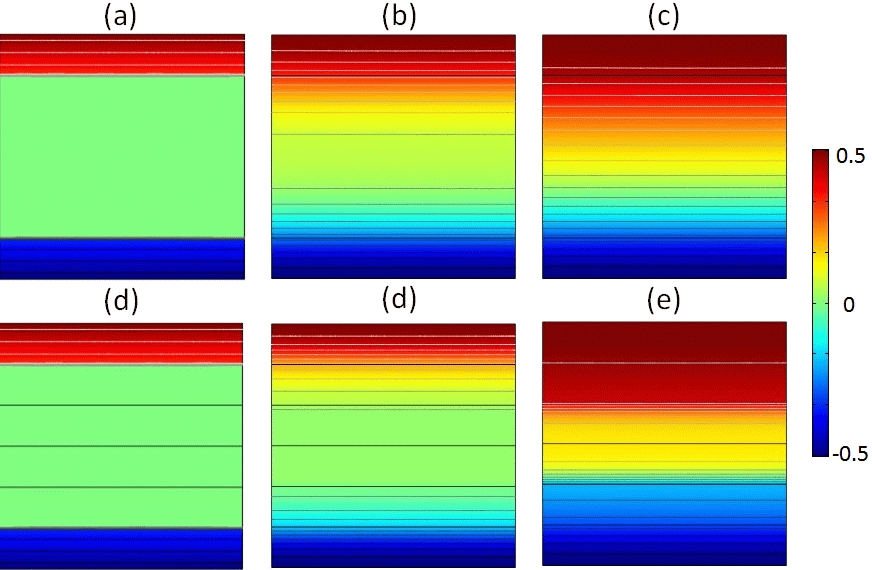}
\caption{Contour plots of a layered cloak (with a source outside $x_2\in(-2,2)$): $f(x)=x_2\sin(x_2)$ for $x\in\{(x_1,x_2):x_2\in(-3,-2)\cup(2,3)\}$ and $f(x)=0$ for $x_2\in(-2,2)$.
Upper panel: Plots of $u$ for data $u^{in}(x)=x_2$ and $g(x)=0$ for $x_2=\pm 3$ and such that $u(-3,x_2)=u(3,x_2)$ for a homogeneous medium with diffusivity $A=1$ at time steps $t=0s$ (a), $1s$ (b) and $4s$ (c). Lower panel: Same for the medium with diffusivity $A=1$ outside $x_2\in(-2,2)$ and a layered cloak inside $x_2\in(-2,2)$ with diffusivity $A(x_2)={\rm diag}\left(\mathbb{A}_{11}(x_2), \frac{1}{\mathbb{A}_{11}(x_2)} \right)$ and density $\rho(x_2)=\mathbb{A}_{11}(x_2)$ with $\mathbb{A}_{11}=\varepsilon$ in $x_2\in(-1,1)$ and $\mathbb{A}_{11}=2-\varepsilon$ in $1<\vert x_2\vert <2$ at time steps $t=0s$ (d), $1s$ (e) and $4s$ (f).}
\label{fig:layered-cloak}
\end{figure}

\begin{figure}[h!]
\includegraphics[width=10cm]{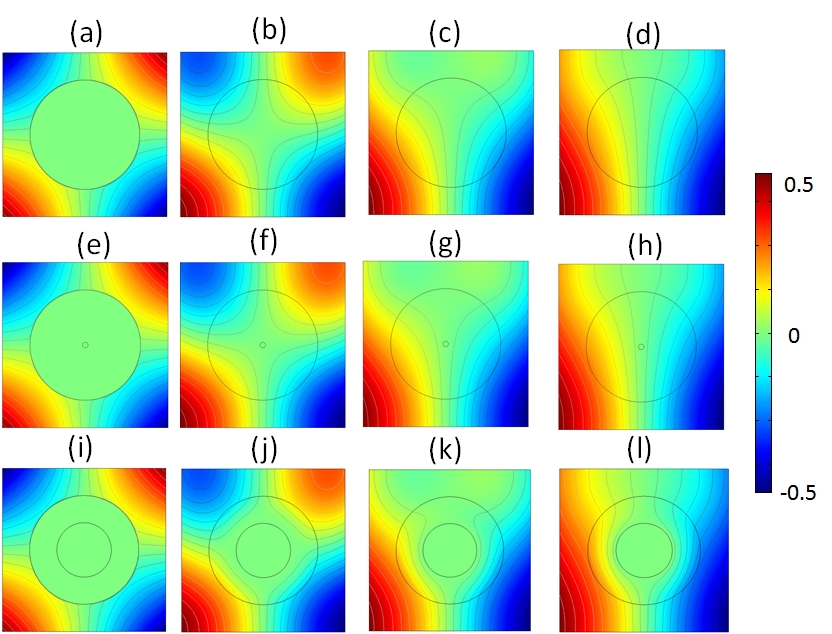}
\caption{Contour plots of a 2D cloak (with a source outside $B_2$): $f(x)=\sqrt{x_1^2+x_2^2}\sin(x_1)\sin(x_2)-2$ for $x\in\Omega\setminus B_2$ and $f(x)=0$ for $x\in B_2$.
Upper panel: Plots of $u$ for data $u^{in}(x_1,x_2)=x_1x_2$ and $g(x_1,x_2)=-3$ for $x_1=\pm 3$ and $0$ for $x_2=\pm 3$ and diffusivity $A=1$ at time steps $t=0s$, $1s$, $4s$ $10s$ and $20s$. Middle panel: same for a small defect of radius $\varepsilon=10^{-1}$, density $\rho^\varepsilon=2\varepsilon^{-2}$ and diffusivity ${A}^\varepsilon=2$ Id. Lower panel: Same for a cloak.}
\label{fig:contour:2D}
\end{figure}

\begin{figure}[h!]
\includegraphics[width=10cm]{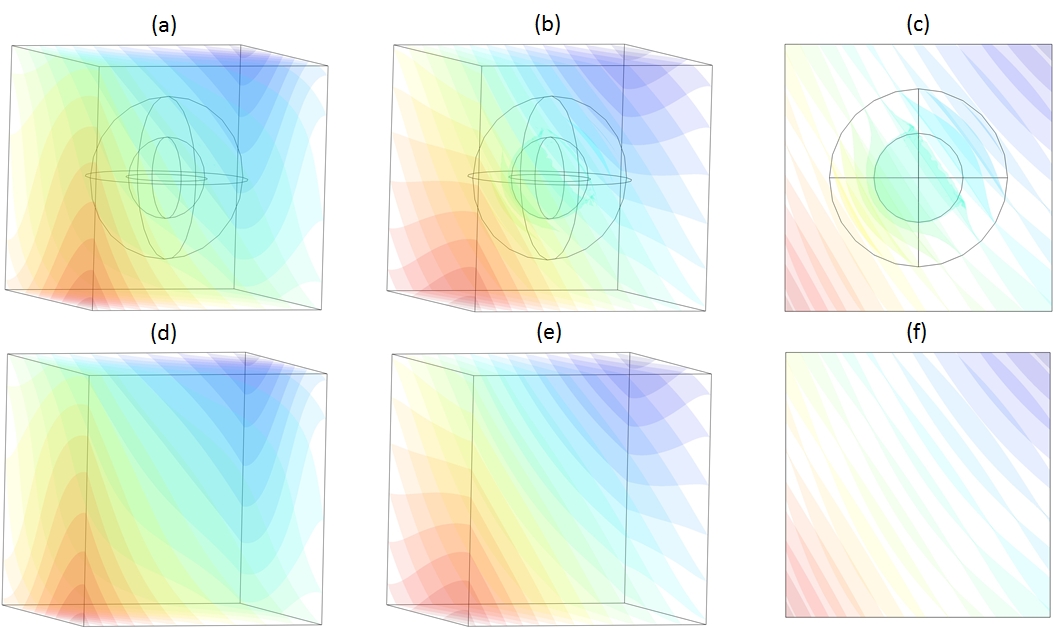}
\caption{Numerical results for isosurface plots (source outside and inside $B_2$): $f_{cl}(x)=\sqrt{x_1^2+x_2^2+x_3^2}\sin(x_1)\sin(x_2)\sin(x_3)$ for $x\in\Omega$.
Upper panel: Plots of $u$ for data $u^{in}(x)=x_3$ and $g(x)=-3$ for $x_1=\pm 3$, and $x_2=\pm 3$ and $0$ for $x_3=\pm 3$ and diffusivity $A=1$ outside a cloak defined as in \eqref{eq:A-cloak-choice} at time steps $t=0.7s$ (a) and $7s$ (b,c). Lower panel: Same for a homogeneous medium with diffusivity $A=1$ at time steps $t=0.7s$ (d) and $7s$ (e,f). Note that (c) and (f) are slices taken in $x_1x_2$-plane at $x_3=0$.}
\label{fig:contour:3D}
\end{figure}

\subsection{Cloaking coefficients}
The cloaking coefficients $\rho_{\rm cl}$ and $A_{\rm cl}$ in the annulus $B_2\setminus B_1$ play all the essential roles in thermal cloaking phenomena. So, it is important in practice -- to gain some physical intuition and start engineering and manufacturing processes of a meta-material cloak -- to analyse the coefficients defined by \eqref{eq:rho-cloak-choice}-\eqref{eq:A-cloak-choice}. For readers' convenience we recall them below (only for the part $\Omega\setminus B_1$ as coefficients inside $B_1$ can be arbitrary)
\begin{equation*}
\rho_{\rm cl}(y) =
\left\{ 
\begin{array}{ll}
1 & \quad \mbox{ for }y\in\Omega\setminus B_2,\\[0.2 cm]
\frac{1}{{\rm det}(D\mathcal{F}_\eps)(x)}\Big|_{x=\mathcal{F}_\eps^{-1}(y)} & \quad \mbox{ for }y\in B_2\setminus B_1
\end{array}\right.
\end{equation*}
\begin{equation*}
A_{\rm cl}(y) =
\left\{ 
\begin{array}{ll}
{\rm Id} & \quad \mbox{ for }y\in\Omega\setminus B_2,\\[0.2 cm]
\frac{D\mathcal{F}_\eps(x) D\mathcal{F}_\eps^\top(x)}{{\rm det}(D\mathcal{F}_\eps)(x)}\Big|_{x=\mathcal{F}_\eps^{-1}(y)} & \quad \mbox{ for }y\in B_2\setminus B_1
\end{array}\right.
\end{equation*}
Remark that both the coefficients $\rho_{\rm cl}$ and $A_{\rm cl}$ depend on the regularising parameter $\eps$ via the Lipschitz map $\mathcal{F}_\eps$. In this numerical test, we plot the cloaking coefficients given in terms of the polar coordinates. Consider the Lipschitz map $\mathcal{F}_\eps:\Omega \mapsto \Omega$
\[
x := (x_1,x_2) \mapsto \left( \mathcal{F}^{(1)}_\eps(x), \mathcal{F}^{(2)}_\eps(x) \right) =: (y_1, y_2) = y
\] 
If the cartesian coordinates $(x_1,x_2)$ were to be expressed in terms of the polar coordinates as $(r\cos\theta, r\sin\theta)$ then the new coordinates $(\mathcal{F}^{(1)}_\eps(x), \mathcal{F}^{(2)}_\eps(x))$ become $(r'\cos\theta, r'\sin\theta)$ with
\begin{equation}\label{eq:r-prime-defn}
r' :=
\left\{
\begin{array}{cl}
r & \mbox{ for } r \ge 2
\\[0.2 cm]
\frac{2-2\eps}{2-\eps} + \frac{r}{2-\eps} & \mbox{ for } \eps < r < 2 
\\[0.2 cm]
\frac{r}{\eps} & \mbox{ for } r \le \eps
\end{array}\right.
\end{equation}
Bear in mind that only the radial coordinate $r$ gets transformed by $\mathcal{F}_\eps$ and the angular coordinate $\theta$ remains unchanged. Reformulating the push-forward maps in terms of the polar coordinates yield the following expressions for the cloaking coefficients in the annulus:
\begin{equation}\label{eq:num:push-forward-polar}
\left.
\begin{aligned}
A_{\rm cl}^{\rm 2D} = \mathcal{F}^\ast\, {\rm Id} & = \mathtt{R}(\theta) {\rm diag}\left(\mathbb{A}_{11}(r'), \frac{1}{\mathbb{A}_{11}(r')} \right) \left[ \mathtt{R}(\theta) \right]^\top
\\
\rho_{\rm cl}^{\rm 2D} = \mathcal{F}^\ast 1 & = \frac{r' -1}{r'}(2-\eps)^2 + \frac{\eps}{r'}(2-\eps)
\end{aligned}
\right\}
\quad \mbox{ for }r'\in(1,2)
\end{equation}
where $\mathbb{A}_{11}(r')$ is given by
\begin{align}\label{eq:mathbb-A11}
\mathbb{A}_{11}(r') = \frac{r' -1}{r'} + \frac{\eps}{r'(2-\eps)}
\qquad \mbox{ for }r'\in[1,2]
\end{align}
and the rotation matrix
\[
\mathtt{R}(\theta) = 
\left(
\begin{matrix}
\cos\theta & -\sin\theta
\\
\sin\theta & \cos\theta
\end{matrix}
\right).
\]
We plot in Figure \ref{fig:num:polar-cloaks} the matrix entry $\mathbb{A}_{11}(r')$ given above and the push-forward density $\rho^{\rm 2D}_{\rm cl}$ given in \eqref{eq:num:push-forward-polar}. We observe that when $\varepsilon$ gets smaller, the radial conductivity and density take values very close to zero near the inner boundary of the cloak, which is unachievable in practice (bear in mind that the azimuthal conductivy being the inverse of the radial conductivity, the conductivity matrix becomes extremely anisotropic). Therefore, manufacturing a meta-material cloak would require a small enough value of  epsilon so that homogenization techniques can be applied to approximate the anisotropic conductivity with concentric layers of isotropic phases, e.g. as was done for the 2D meta-material cloak manufactured at the Karlsruher Institut f\"ur Technologie \cite{Schittny_2013}.
\begin{figure}[h!]
\includegraphics[width=12cm]{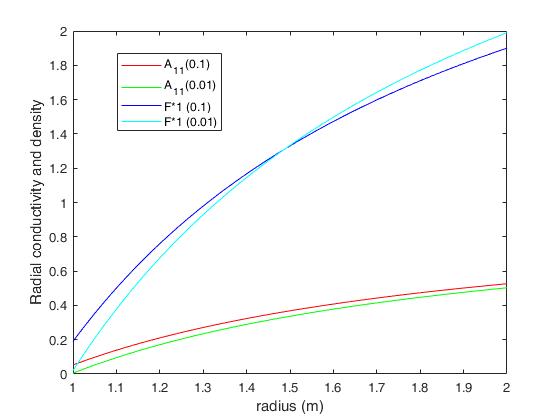}
\caption{Plots of $\rho^{\rm 2D}_{\rm cl}(r')$ and $\mathbb{A}_{11}(r')$ for $r'\in[1,2]$ and $\varepsilon=10^{-1}, 10^{-2}$.}
\label{fig:num:polar-cloaks}
\end{figure}

In three spatial dimensions, we can recast the cloaking coefficients in the spherical coordinates $(r,\theta,\varphi)$. As in the cylindrical coordinate setting, only the radial variable gets modified by the Kohn's transformation $\mathcal{F}^\eps$ and the variables $\theta$, $\varphi$ remain unchanged. The transformed radial coordinate $r'$ is given by \eqref{eq:r-prime-defn}. The push-forward maps of interest for cloaking are
\begin{equation}\label{eq:num:push-forward-spherical}
\left.
\begin{aligned}
A_{\rm cl}^{\rm 3D} = \mathcal{F}^\ast\, {\rm Id} & = \mathtt{R}(\theta) \mathtt{M}(\varphi) {\rm diag}\Big(\mathbb{B}(r'), 2-\eps, 2-\eps \Big) \mathtt{M}(\varphi) \left[ \mathtt{R}(\theta) \right]^\top
\\
\rho_{\rm cl}^{\rm 3D} = \mathcal{F}^\ast 1 & = \mathbb{B}(r') := (2-\eps) \left( 2-\eps - \frac{(2-2\eps)}{r'}\right)^2
\end{aligned}
\right\}
\quad \mbox{ for }r'\in(1,2)
\end{equation}
with the rotation matrix $\mathtt{R}(\theta)$ in three dimensions and the matrix $\mathtt{M}(\varphi)$ given by
\[
\mathtt{R}(\theta) = 
\left(
\begin{matrix}
\cos\theta & -\sin\theta & 0
\\
\sin\theta & \cos\theta & 0
\\
0 & 0 & 1
\end{matrix}
\right)
\quad
\mathtt{M}(\varphi) = 
\left(
\begin{matrix}
\sin\varphi & 0 & \cos\varphi
\\
0 & 1 & 0
\\
\cos\varphi & 0 & -\sin\varphi
\end{matrix}
\right)
\]
Note that the matrix entry $\mathbb{B}(r')$ in the push-forward conductivity coincides with the push-forward density $\rho_{\rm cl}^{\rm 3D}$. We plot in Figure \ref{fig:num:spherical-cloaks} the radial conductivity and density in the cloaking annulus given in \eqref{eq:num:push-forward-spherical}. One should recall that the 3D spherical cloak only has a varying radial conductivity, the other two polar and azimuthal diagonal entries of the conductivity matrix being constant (i.e. independent of the radial position). Besides, the radial conductivity has the same value as the density within the cloak. All this makes the 3D spherical cloak easier to approximate with homogenization techniques. However, no thermal cloak has been manufactured and experimentally characterised thus far, perhaps due to current limitations in 3D manufacturing techniques; a possible route towards construction of a spherical cloak is 3D printing. Similar computations in spherical coordinates, but for Pendry's singular transformation are found in \cite{Nicolet_2008}, where the matrix $\mathtt{M}$ was introduced to facilitate implementation of perfectly matched layers, cloaks and other transformation based electromagnetic media with spherical symmetry in finite element packages, see also \cite{Nicolet_1994} that predates the field of transformational optics.

\begin{figure}[h!]
\includegraphics[width=12cm]{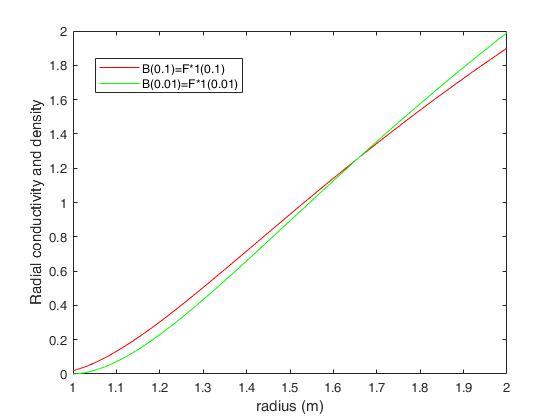}
\caption{Plots of $\rho^{\rm 3D}_{\rm cl}(r')=\mathbb{B}(r')$ for $r'\in[1,2]$ and $\varepsilon=10^{-1}, 10^{-2}$.}
\label{fig:num:spherical-cloaks}
\end{figure}

\subsection{Spectral problems} Here we perform some numerical tests in support of the spectral result (Proposition \ref{prop:first-non-zero-eigenvalue}) proved in this paper. The two Neumann spectral problems that we study are
\[
-\Delta \phi = \mu\, \phi \quad \mbox{ in }\Omega:=(-3,3)^d,
\quad
\nabla\phi\cdot{\bf n} = 0 \quad \mbox{ on }\partial\Omega.
\]
\[
-{\rm div}\Big(A^\eps(x) \nabla \phi^\eps \Big) = \mu\, \rho^\eps(x)\, \phi^\eps \quad \mbox{ in }\Omega:=(-3,3)^d,
\quad
\nabla\phi^\eps\cdot{\bf n} = 0 \quad \mbox{ on }\partial\Omega.
\]
The coefficients are of high contrast and take the form
\begin{equation*}
\rho^\eps(x), \, A^\eps(x) =
\left\{
\begin{array}{ll}
1, \qquad \quad {\rm Id} & \quad \mbox{ for }x\in\Omega\setminus B_\eps,\\[0.2 cm]
\frac{1}{\eps^d}, \qquad \frac{1}{\eps^{d-2}}{\rm Id} & \quad \mbox{ for }x\in B_\eps
\end{array}\right.
\end{equation*}
The first non-zero eigenvalue for the Neumann Laplacian is $\mu_{2}={(\pi/6)}^{d}$. We illustrate the result of Proposition \ref{prop:first-non-zero-eigenvalue} by showing that the first non-zero eigenvalue $\mu^\eps_2$ for the defect spectral problem is $\eps^d$-close to $\mu_2$ for various value of $\eps$ and in dimensions one, two and three. The results are tabulated in Table \ref{table:eigenprob}. The associated eigenfields are also plotted in Figure \ref{fig:eigenfields}. Note that the spectral problems in 1D, 2D and 3D are solved with direct UMFPACK and PARDISO solvers using adaptive meshes with 250, 150 and 50 thousand elements, respectively. We made sure that there are at least $10^2$ elements in the small defect for every spectral problem solved.

\begin{table}
\begin{tabular}{|l|l|l|l|c|c|c|c|c|c|r|r|r|r|}
   \hline
   &  \multicolumn{4}{|l|}{Numerical illustration of proposition 5}  &\\
   \hline
  &  dim. $d$   & $\mu_{2}={(\pi/6)}^{d}$   &  $\mu_{2}^\varepsilon$
  & $\mid\mu_{2}-\mu_{2}^\varepsilon\mid$ & Parameter $\varepsilon$\\
   \hline 
 & 1  & 0.52359877559  & 0.5342584732    &    0.0106596976                             &$\varepsilon =10 ^{-1}$
 \\
Numerical
 & 1 & 0.52359877559 & 0.5327534635     &  0.0091546879                & $\varepsilon =10 ^{-2}$  
\\
 validation of
  & 1  & 0.52359877559  & 0.5245077454     &  0.0009896980                  &  $\varepsilon =10 ^{-3}$
 \\
 & 1  & 0.52359877559  & 0.5236795221   & 0.0000807465    &  $\varepsilon =10 ^{-4}$
 \\
  $\mid\mu_{2}-\mu_{2}^\varepsilon\mid \leq \varepsilon^d$
 & 1   &  0.52359877559  & 0.5236081655    & 0.0000093899   &  $\varepsilon =10 ^{-5}$
\\
 & 1  & 0.52359877559 &  0.52359897453    & 0.0000001989 &  $\varepsilon =10 ^{-6}$
 \\
 & 1  & 0.52359877559  & 0.52359888455     & 0.0000001089  &  $\varepsilon =10 ^{-7}$ 
\\
   \hline
 & 2  & 0.27415567781 & 0.2741626732 & 0.0000069954 & $\varepsilon =10 ^{-1}$
 \\
 & 2 & 0.27415567781 & 0.2741546789 &  0.0000009989 & $\varepsilon =10 ^{-2}$  
\\
  & 2  & 0.27415567781  &  0.2741556795    & 0.0000000017 &  $\varepsilon =10 ^{-3}$
  \\
 \hline
 & 3  & 0.14354757722 & 0.1437347845    & 0.0001872072                                 &$\varepsilon =10 ^{-1}$
 \\
 \hline
\end{tabular}
\label{table:eigenprob}
\caption{ Numerical estimate  of the difference $\mid\mu_{2}-\mu_{2}^\varepsilon\mid$, versus the parameter $\varepsilon =10^{-m}$ with $m=1,\cdots, 7$ for dimension $d=1$, with $m=1,\cdots, 3$ for $d=2$ and with $m=1$ for $d=3$. Same source, Neumann data, diffusivity and density parameters as in Figure \ref{fig:eigenfields}.}
\end{table}

\begin{figure}[h!]
\includegraphics[width=10cm]{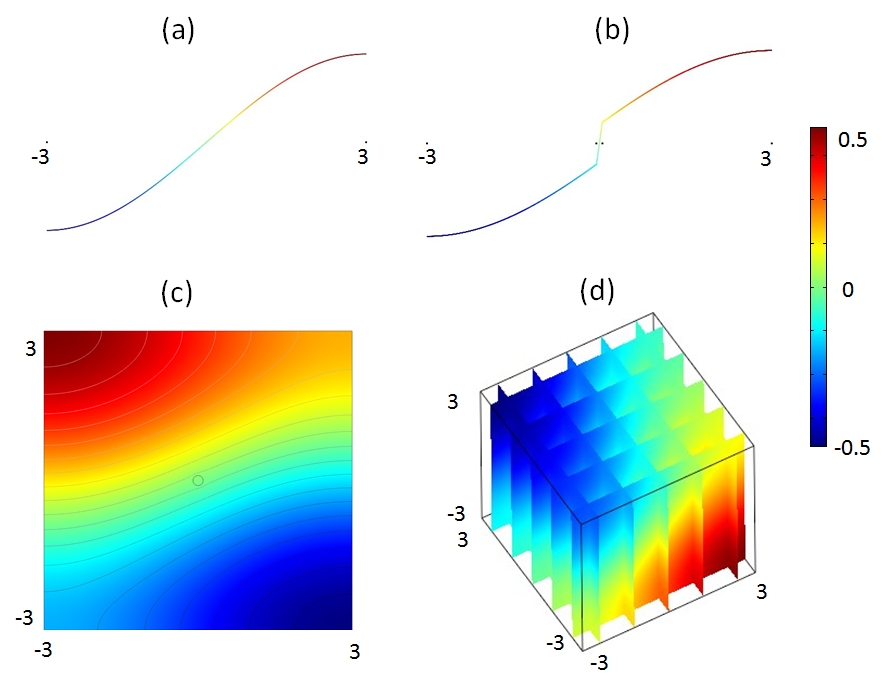}
\caption{Numerical results for contour plots of eigenfields $\phi$ and $\phi^\varepsilon$ associated with
$\mu_{2}$ (a) and $\mu_{2}^{\varepsilon}$ (b-d), in dimension $1$ (a,b), $2$ (c) and $3$ (d).
}
\label{fig:eigenfields}
\end{figure}

\section{Concluding remarks}
This work addressed the question of near-cloaking in the time-dependent heat conduction problem. The main inspiration is derived from the work of Kohn and co-authors \cite{Kohn_2008} which quantified near-cloaking in terms of certain boundary measurements. Hence the main result of this paper (see Theorem \ref{thm:near-cloak}) asserts that the difference between the solution to the cloak problem \eqref{eq:intro-u_cloak} and that to the homogeneous problem \eqref{eq:intro-u_hom} when measured in the $\mathrm H^\frac12$-norm on the boundary can be made as small as one wishes by fine-tuning certain regularisation parameter. To the best of our knowledge, this is the first work to consider near-cloaking strategies to address time-dependent heat conduction problem. This work supports the idea of thermal cloaking albeit with the price of having to wait for certain time to see the effect of cloaking. We also illustrate our theoretical results by some numerical simulations. We leave the study of fine properties of the thermal cloak problem for future investigations:
\begin{itemize}
\item Behaviour of the temperature field inside the cloaked region.
\item Designing certain multi-scale structures (\`a la reiterated homogenization) to achieve effective properties close to the characteristics of $\rho_{\rm cl}$ and $A_{\rm cl}$.
\item Study thermal cloaking for time-harmonic sources.
\end{itemize}

\noindent{\bf Acknowledgments.} 
The authors would like to thank Yves Capdeboscq for helpful discussions regarding the Calder\'on inverse problem and for bringing to our attention the work of Kohn and co-authors \cite{Kohn_2008}.
H.H. and R.V.C. acknowledge the support of the EPSRC programme grant ``Mathematical fundamentals of Metamaterials for multiscale Physics and Mechanics'' (EP/L024926/1).
R.V.C. also thanks the Leverhulme Trust for their support. 
S.G. acknowledges support from EPSRC as a named collaborator on grant EP/L024926/1. 
G.P. acknowledges support from the EPSRC via grants EP/L025159/1, EP/L020564/1, EP/L024926/1, EP/P031587/1.


\bibliography{references}

\end{document}